\documentclass[a4paper,12pt]{amsart}
\usepackage{amsfonts}
\usepackage{amsmath}
\usepackage{hyperref}
\usepackage{amssymb}
\usepackage{color,tikz}
\usepackage[normalem]{ulem}

\allowdisplaybreaks
\numberwithin{equation}{section}
\newtheorem{theorem}{Theorem}[section]
\newtheorem{proposition}[theorem]{Proposition}
\newtheorem{lemma}[theorem]{Lemma}

\theoremstyle{definition}
\newtheorem{definition}[theorem]{Definition}

\theoremstyle{remark}
\newtheorem{remark}[theorem]{Remark}
\newtheorem{example}[theorem]{Example}

\begin{document}
\title{Cohn path algebras of higher-rank graphs}
\author{Lisa Orloff Clark and Yosafat E. P. Pangalela}
\address{Lisa Orloff Clark and Yosafat E. P. Pangalela\\
Department of Mathematics and Statistics\\
University of Otago\\
PO Box 56\\
Dunedin 9054\\
New Zealand}
\email{lclark@maths.otago.ac.nz, yosafat.pangalela@maths.otago.ac.nz}

\begin{abstract}
In this article, we introduce Cohn path algebras of higher-rank graphs. We
prove that for a higher-rank graph $\Lambda $, there exists a higher-rank
graph $T\Lambda $ such that the Cohn path algebra of $\Lambda $ is
isomorphic to the Kumjian-Pask algebra of $T\Lambda $. We then use this
isomorphism and properties of Kumjian-Pask algebras to study Cohn path
algebras. This includes proving a uniqueness theorem for Cohn path algebras.
\end{abstract}

\thanks{This research was done as part of the second author's PhD thesis at
the University of Otago under the supervision of the first author and Iain
Raeburn. Thank you to Iain for his guidance. This research was also
supported by Marsden grant 15-UOO-071 from the Royal Society of New Zealand.}
\maketitle

\section{Introduction}

Leavitt path algebras were introduced and studied in \cite{AA05} and \cite%
{AMP07} as a generalisation the class of algebras studied by Leavitt in \cite%
{L62}. Leavitt path algebras are also the natural algebraic analogues of
graph $C^{\ast }$-algebras studied in \cite{CBMS}; a number of interesting
result have been proven by translating between graph $C^{\ast }$-algebras
and the ring theoretic Leavitt path algebras.

Cohn path algebras were introduced in \cite{AM12,AG11} and generalise the
algebras $U_{1,n}$ studied by Cohn in \cite{C55}. The idea is to build an
algebra out of the path space of a graph; addition and scalar multiplication
are defined formally and multiplication of two paths is only nonzero when
one can concatenate the paths. Cohn path algebras can also be obtained from
Leavitt path algebras by omitting one of the \emph{Cuntz-Krieger relations}.
Hence every Leavitt path algebra can be viewed as a quotient of a Cohn path
algebra. On the other hand, Abrams, Ara and Siles Molina show in \cite%
{Leavitt path algebras} that for every graph $E$, there exists a graph $TE$
(denoted $E\left( X\right) $ in \cite{Leavitt path algebras}) such that the
Cohn path algebra of $E$ is isomorphic to the Leavitt path algebra of $TE$.

Although it has received less attention, the Cohn path algebra of a directed
graph is the algebraic analogue of the $C^{\ast}$-algebraic \emph{Toeplitz
algebra} of $E$ as defined in \cite{FR99}. In the $C^{\ast }$-algebraic
setting, Muhly and Tomforde in \cite{MT04} and also Sims in \cite{Si10} each
show that for a graph $E$, the Toeplitz algebra of $E$ is isomorphic to the
graph $C^{\ast }$-algebra of $TE$.

Now we move into the setting of higher-rank graph algebras: In \cite{KP00},
Kumjian and Pask introduced a combinatorial model, called a higher-rank
graph, in order to capture the essential features of the $C^{\ast }$%
-algebras studied by Robertson and Steger in \cite{RS02}. A higher-rank
graph, also called a $k$-graph, is a generalisation of the path category of
a directed graph where the length of a `path' $\lambda$ in a $k$-graph is an
element of $\mathbb{N}^k$. Kumjian and Pask studied the $C^{\ast }$-algebras
associated to \emph{row-finite} higher-rank graphs with \emph{no sources}.
Raeburn, Sims and Yeend generalised Kumjian and Pask's construction by
describing the class of $C^{\ast }$-algebras associated to more general
higher-rank graphs in \cite{RSY03,RSY04}.

A few years ago, a higher-rank analogue of Leavitt path algebras, called
Kumjian-Pask algebras, was introduced in \cite{ACaHR13}. The class of
Kumjian-Pask algebras includes the class of Leavitt path algebras. In \cite%
{ACaHR13} the authors limit their focus to row-finite higher-rank graphs
with no sources. Following the generalisation pattern of higher-rank graph $%
C^{\ast } $-algebras, Kumjian-Pask algebras associated to more general
higher-rank graphs are described in \cite{CFaH14} and \cite{CP15}.

$C^*$-algebraic Toeplitz algebras of higher-rank graphs were introduced in
\cite{RS05}. Thus it seems natural to ask whether there is also an algebraic
analogue of these $C^{\ast}$-algebras. In this paper, we introduce Cohn path
algebras of higher-rank graphs. Our motivation comes from a desire to one
day establish an algebraic version of `KMS\ states' for higher-rank graph
algebras (see \cite{aHKR15,aHLRS14,aHLRS15}).

Our strategy is to follow the analysis of \cite{P15}. In that paper,
Pangalela shows that for every row-finite higher-rank graph $\Lambda $,
there exists a higher-rank graph $T\Lambda $ such that the Toeplitz algebra
of $\Lambda $ is isomorphic to the $C^{\ast }$-algebra of $T\Lambda $.
Although we will start with a row-finite $k$-graph $\Lambda $ with no
sources, the $k$-graph $T\Lambda $ always has sources and is not `locally
convex' so we will need to use the Kumjian-Pask algebra construction given
in \cite{CP15}.

Let $\Lambda $ be a row-finite higher-rank graph with no sources and $R$ be
a commutative ring with 1. After providing some preliminaries, in Section %
\ref{Section_CPfamily}, we define a Cohn $\Lambda $-family (\ref%
{CP-family}) and show there exists a universal Cohn path algebra $%
{\normalsize \operatorname{C}}_{R}\left( \Lambda \right) $ (Proposition \ref%
{universal-CP-family}).

In Section \ref{Section-the-uniqueness-theorem-of-CP-family}, we recall the
Kumjian-Pask algebras of \cite{CP15} and the higher-rank graph $T\Lambda $
of \cite{P15}. We also study properties of the Kumjian-Pask algebra of $%
T\Lambda $ (Proposition \ref{KP-TLambda-when-lambda-has-no-sources}) and
show that the Cohn path algebra of $\Lambda $ is isomorphic to the
Kumjian-Pask algebra of $T\Lambda $ (Theorem \ref{CP-is-isomorphic-to-KP}).
This isomorphism is an algebraic version of \cite[Theorem 4.1]{P15}. We then
show that every Cohn path algebra is $\mathbb{Z}^{k}$-graded (see \cite[%
Theorem 3.6]{CP15}). At the end of the section, we use the Cuntz-Krieger
uniqueness theorem for Kumjian-Pask algebras in \cite[Theorem 8.1]{CP15} to
prove the uniqueness theorem for Cohn path algebras (Theorem \ref%
{the-uniqueness-theorem-of-CP-family}).

Our uniqueness theorem for Cohn path algebras is notable for two reasons.
The first is that it gives a uniqueness theorem for Cohn path algebras
associated to directed graphs. Although in \cite{Leavitt path algebras},
Abrams, Ara and Siles Molina prove that every Cohn path algebra is a Leavitt
path algebra and state the Cuntz-Krieger uniqueness theorem for Leavitt path
algebras, they do not explicitly investigate uniqueness theorems for Cohn
path algebras.

Secondly, we can view the uniqueness theorem for Cohn path algebras as an
algebraic analogue of the uniqueness theorem for Toeplitz algebras given in
\cite{RS05}; the proof of our algebraic uniqueness theorem is considerably
shorter than the one Toeplitz algebras in \cite{RS05}. By translating our
proof into the $C^{\ast }$-algebra setting, we provide an alternative proof
of the uniqueness theorem for Toeplitz algebras (see \cite[Remark 4.3]{P15}).

Finally, we discuss examples and applications in Section \ref%
{Section-examples-and-applications}. First we explicitly demonstrate the
relationship between Cohn path algebras and Toeplitz algebras (Proposition %
\ref{CP-is-dense-in-TC}). We also show that our Cohn algebras can be
realised as Steinberg algebras (Proposition \ref%
{CP-is-isomorphic-to-Steinberg-algebras}).

\section{Preliminaries}

Let $k$ be a positive integer. We regard $\mathbb{N}^{k}$ as an additive
semigroup with identity $0$. For $n\in \mathbb{N}^{k}$, we write $n=\left(
n_{1},\ldots ,n_{k}\right) $. Meanwhile, for $m,n\in \mathbb{N}^{k}$, we
write $m\leq n$ to denote $m_{i}\leq n_{i}$ for $1\leq i\leq k$, and we use
expression $m\vee n$ for their coordinate-wise maximum and $m\wedge n$ for
their coordinate-wise minimum. We also write $e_{i}$ for the usual basis
elements in $\mathbb{N}^{k}$.

\subsection{Higher-rank graphs.}

A \emph{higher-rank graph} or $k$\emph{-graph} $\Lambda =\left( \Lambda
^{0},\Lambda ,r,s\right) $ is a countable small category $\Lambda $ with a
functor $d$ from $\Lambda $ to $\mathbb{N}^{k}$, called the \emph{degree map}%
, which satisfies the \emph{factorisation property}: for every $\lambda \in
\Lambda $ and $m,n\in \mathbb{N}^{k}$ with $d\left( \lambda \right) =m+n$,
there exist unique elements $\mu ,\nu \in \Lambda $ such that $\lambda =\mu
\upsilon $ and $d\left( \mu \right) =m$, $d\left( \nu \right) =n$. We then
write $\lambda \left( 0,m\right) $ for $\mu $ and $\lambda \left(
m,m+n\right) $ for $\nu $. Note that $\lambda \mu $ denotes the composition
of paths $\lambda $ and $\mu $ with $s\left( \lambda \right) =r\left( \mu
\right) $. We call the elements of $\Lambda$ \emph{paths} and elements of $%
\Lambda^0$ \emph{vertices}.

For $k=1$, we use notation $E=\left( E^{0},E^{\ast },r,s\right) $ to denote
a $1$-graph. In this case, $E^{\ast }$ contains all paths in $E$ with degree
at least $1$. We also write $E^{1}$ for the set of all paths with degree $1$%
. Since we view $E$ as a category, we use different convention from that of
Leavitt path algebra and Cohn path algebra literature where people write $%
\lambda \mu $ to denote the composition of paths $\lambda $ and $\mu $ with $%
s\left( \mu \right) =r\left( \lambda \right) $.

One way to visualise $k$-graphs is to use coloured directed graphs, as
described in \cite{HRSW13}. Suppose $\Lambda $ is a $k$-graph. The coloured
graph associated to $\Lambda$ is a directed graph whose edges are colour
coded: Choose $k$-different colours $c_{1},\ldots ,c_{k}$. The vertices in
the coloured graph are the same as the vertices of $\Lambda$. Each path $%
\lambda$ in $\Lambda$ with degree $e_{i}$ corresponds to an edge of colour $%
c_{i}$ between $s(\lambda)$ and $r(\lambda)$. We call this coloured graph
the \emph{skeleton} of $\Lambda$.

\begin{example}[{\protect\cite[Example 2.2.(ii)]{RSY03}}]
Let $k\in \mathbb{N}$ and $n\in \left( \mathbb{N\cup }\left\{ \infty
\right\} \right) ^{k}$. We define%
\begin{equation*}
\Omega _{k,n}:=\left\{ \left( p,q\right) \in \mathbb{N}^{k}\times \mathbb{N}%
^{k}:p\leq q\leq n\right\} \text{.}
\end{equation*}%
This is a category with objects $\left\{ p\in \mathbb{N}^{k}:p\leq n\right\}
$, range map $r\left( p,q\right) =p$, source map $s\left( p,q\right) =q$,
and degree map $d\left( p,q\right) =q-p$. Then $\Omega _{k,n}$ is a $k$%
-graph. The skeleton of $\Omega _{2,\left( 1,2\right) }$ is%
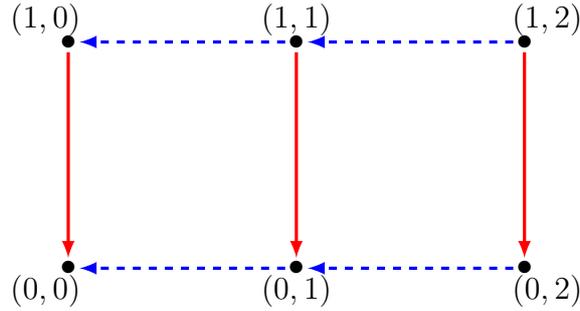
\begin{equation*}
\begin{tikzpicture} \node[inner sep=1pt] (v_1) at (0,0) {$\bullet$};
\node[inner sep=1pt] at (-0.3,-0.3) {$(0,0)$}; \node[inner sep=1pt] (v_2) at
(0,3) {$\bullet$}; \node[inner sep=1pt] at (-0.3,3.3) {$(1,0)$}; \node[inner
sep=1pt] (v_3) at (3,0) {$\bullet$}; \node[inner sep=1pt] at (3,-0.3)
{$(0,1)$}; \node[inner sep=1pt] (v_4) at (3,3) {$\bullet$}; \node[inner
sep=1pt] at (3,3.3) {$(1,1)$}; \node[inner sep=1pt] (v_5) at (6,0)
{$\bullet$}; \node[inner sep=1pt] at (6.3,-0.3) {$(0,2)$}; \node[inner
sep=1pt] (v_6) at (6,3) {$\bullet$}; \node[inner sep=1pt] at (6.3,3.3)
{$(1,2)$}; \draw[-latex, red, very thick] (v_2) edge[out=270, in=90](v_1);
\draw[-latex, blue, dashed, very thick] (v_4) edge[out=180, in=0](v_2);
\draw[-latex, red, very thick] (v_4) edge[out=270, in=90] (v_3);
\draw[-latex, blue, dashed, very thick] (v_3) edge[out=180, in=0] (v_1);
\draw[-latex, red, very thick] (v_6) edge[out=270, in=90] (v_5);
\draw[-latex, blue, dashed, very thick] (v_5) edge[out=180, in=0] (v_3);
\draw[-latex, blue, dashed, very thick] (v_6) edge[out=180, in=0] (v_4);
\end{tikzpicture}
\end{equation*}
where solid edges have degree $\left( 1,0\right) $ and dashed edges have
degree $\left( 0,1\right) $.
\end{example}

We write
\begin{equation*}
W_{\Lambda }:=\bigcup_{n\in \left( \mathbb{N\cup }\left\{ \infty \right\}
\right) ^{k}}\{x:\Omega _{k,n}\rightarrow \Lambda :x\text{ is a degree
preserving functor}\}\text{.}
\end{equation*}%
Suppose $x\in W_{\Lambda }$. For $n\in \mathbb{N}^{k}$ and $n\leq d\left(
x\right) $, the path $\sigma ^{n}x$ is defined by $\sigma ^{n}x\left(
0,m\right) =x\left( n,n+m\right) $ for all $m\leq d\left( x\right) -n$.

For $n\in \mathbb{N}^{k}$, we define%
\begin{equation*}
\Lambda ^{n}:=\{\lambda \in \Lambda :d\left( \lambda \right) =n\}
\end{equation*}%
and call the elements $\lambda $ of $\Lambda ^{n}$ \emph{paths of degree }$n$%
. In particular, we regard elements of $\Lambda ^{0}$ as \emph{vertices}.\
We use term \emph{edge} to denote a path $e\in \Lambda ^{e_{i}}$ where $%
1\leq i\leq k$, and write
\begin{equation*}
\Lambda ^{1}:=\bigcup_{1\leq i\leq k}\Lambda ^{e_{i}}
\end{equation*}%
for the set of all edges.

For $v\in \Lambda ^{0}$, $\lambda \in \Lambda $ and $E\subseteq \Lambda $,
we define
\begin{equation*}
vE:=\left\{ \mu \in E:r\left( \mu \right) =v\right\} \text{ and }\lambda
E:=\left\{ \lambda \mu \in \Lambda :\mu \in E,r\left( \mu \right) =s\left(
\lambda \right) \right\} \,\text{.}
\end{equation*}%
We say that $\Lambda $ is \emph{row-finite} if for every $v\in \Lambda ^{0}$%
, the set $v\Lambda ^{e_{i}}$ is finite for $1\leq i\leq k$. Finally, we say
$v\in \Lambda ^{0}$ is a \emph{source} if there exists $m\in \mathbb{N}^{k}$
such that $v\Lambda ^{m}=\emptyset $.

\begin{example}
Consider the $2$-graph $\Lambda _{1}$ which has skeleton
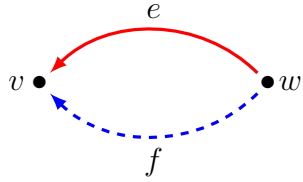
\begin{equation*}
\begin{tikzpicture} \node[inner sep=1pt] (v) at (0,0) {$\bullet$};
\node[inner sep=1pt] at (-0.3,0) {$v$}; \node[inner sep=1pt] (w) at (3,0)
{$\bullet$}; \node[inner sep=1pt] at (3.3,0) {$w$}; \draw[-latex, red, very
thick] (w) edge[out=135, in=45] node[pos=0.5, above, black]{$e$} (v);
\draw[-latex, blue, dashed, very thick] (w) edge[out=225, in=315]
node[pos=0.5, below, black]{$f$} (v); \end{tikzpicture}
\end{equation*}
where the solid edge has degree $\left( 1,0\right) $ and the dashed edge has
degree $\left( 0,1\right) $. It is clear that $w$ is a source since there is
no paths going in the vertex. Hence, $\Lambda _{1}$ is row-finite with
sources.
\end{example}

\begin{example}
Let $\Lambda _{2}$ be the $2$-graph with skeleton%
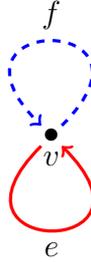
\begin{equation*}
\begin{tikzpicture} \node[inner sep=1pt] (v) at (0,0) {$\bullet$};
\node[inner sep=1pt] at (0,-0.3) {$v$}; \path[->,every
loop/.style={looseness=20}] (v) edge[in=-45,out=-135,loop, red, very thick]
node[pos=0.5, below, black]{$e$} (v); \path[->,every
loop/.style={looseness=20}] (v) edge[in=-225,out=-315,loop, blue, dashed,
very thick] node[pos=0.5, above, black]{$f$} (v); \end{tikzpicture}
\end{equation*}
where $ef=fe$, the solid edge has degree $\left( 1,0\right) $ and the dashed
edge has degree $\left( 0,1\right) $. Since $v\Lambda ^{m}\neq \emptyset $
for all $m\in \mathbb{N}^{k}$, then $\Lambda _{2}$ is row-finite with no
sources.
\end{example}

For $\lambda ,\mu \in \Lambda $, we define%
\begin{equation*}
\operatorname{MCE}\left( \lambda ,\mu \right) :=\left\{ \tau \in \Lambda :d\left(
\tau \right) =d\left( \lambda \right) \vee d\left( \mu \right) ,\tau \left(
0,d\left( \lambda \right) \right) =\lambda ,\tau \left( 0,d\left( \mu
\right) \right) =\mu \right\}
\end{equation*}%
and%
\begin{equation*}
\Lambda ^{\min }\left( \lambda ,\mu \right) :=\left\{ \left( \lambda
^{\prime },\mu ^{\prime }\right) \in \Lambda \times \Lambda :\lambda \lambda
^{\prime }=\mu \mu ^{\prime }\in \operatorname{MCE}\left( \lambda ,\mu \right)
\right\} \text{.}
\end{equation*}%
Meanwhile, for $E\subseteq \Lambda $ and $\lambda \in \Lambda $, we write
\begin{equation*}
\operatorname{Ext}\left( \lambda ;E\right) :=\bigcup_{\mu \in E}\left\{ \rho :(\rho
,\tau )\in \Lambda ^{\min }\left( \lambda ,\mu \right) \right\} \text{.}
\end{equation*}%
A set $E\subseteq v\Lambda $ is \emph{exhaustive} if for all $\lambda \in
v\Lambda $, there exists $\mu \in E$ such that $\Lambda ^{\min }\left(
\lambda ,\mu \right) \neq \emptyset $.

In this article, we focus on row-finite $k$-graphs. For further discussion
about row-finite $k$-graphs and their generalisations, see \cite%
{KP00,RSY03,RSY04,CBMS,W11}.

\subsection{Graded rings.}

Let $G$ be an additive abelian group. If $A$ is a ring, we say that $A$ is $%
G $\emph{-graded }if there are additive subgroups $\left\{ A_{g}:g\in
G\right\} $ satisfying:%
\begin{equation*}
A=\bigoplus {}_{g\in G}A_{g}\text{ and for }g,h\in G\text{, }%
A_{g}A_{h}\subseteq A_{g+h}\text{.}
\end{equation*}%
For $g\in G$, the subgroup $A_{g}$ is called the \emph{homogeneous component
of }$A$\emph{\ of degree }$g$.

\section{Cohn $\Lambda $-families}

\label{Section_CPfamily}Throughout this section, suppose that $\Lambda $ is
a row-finite $k$-graph with no sources and $R$ is a commutative ring with $1$%
. For each $\lambda \in \Lambda $, we introduce a formal symbol $\lambda
^{\ast }$ called a \emph{ghost path}; if $v\in \Lambda ^{0}$, we identify $%
v^{\ast }:=v$. We then write $G\left( \Lambda \right) $ the set of ghost
paths and define $r$ and $s$ on $G\left( \Lambda \right) $ by%
\begin{equation*}
\text{ }r\left( \lambda ^{\ast }\right) :=s\left( \lambda \right) \text{ and
}s\left( \lambda ^{\ast }\right) :=r\left( \lambda \right) \text{.}
\end{equation*}%
We also define composition in $G\left( \Lambda \right) $: set $\lambda
^{\ast }\mu ^{\ast }=\left( \mu \lambda \right) ^{\ast }$ for $\lambda ,\mu
\in \Lambda $ with $s(\mu )=r(\lambda )$; and finally, we write $G\left(
\Lambda ^{\neq 0}\right) :=\left\{ \lambda ^{\ast }:\lambda \in \left.
\Lambda \right\backslash \Lambda ^{0}\right\} $.

\begin{definition}
\label{CP-family}A \emph{Cohn }$\Lambda $\emph{-family} $\left\{ T_{\lambda
},T_{\mu ^{\ast }}:\lambda ,u\in \Lambda \right\} $ in an $R$-algebra $A$
consists of a map $T:\Lambda \cup G\left( \Lambda ^{\neq 0}\right)
\rightarrow A$ such that:
\end{definition}

\begin{enumerate}
\item[(CP1)] $\left\{ T_{v}:v\in \Lambda ^{0}\right\} $ is a collection of
mutually orthogonal idempotents;

\item[(CP2)] for $\lambda ,\mu \in \Lambda $ with $s\left( \lambda \right)
=r\left( \mu \right) $, we have $T_{\lambda }T_{\mu }=T_{\lambda \mu }$ and $%
T_{\mu ^{\ast }}T_{\lambda ^{\ast }}=T_{\left( \lambda \mu \right) ^{\ast }}$%
;

\item[(CP3)] $T_{\lambda ^{\ast }}T_{\mu }=\sum_{(\nu ,\gamma )\in \Lambda
^{\min }\left( \lambda ,\mu \right) }T_{\nu }T_{\gamma ^{\ast }}$ for all $%
\lambda ,\mu \in \Lambda $.
\end{enumerate}

\begin{remark}
\begin{enumerate}
\item[(i)] For $1$-graph $E$, people usually write $\left\{ v,e,e^{\ast
}:v\in E^{0},e\in E^{1}\right\} $ instead of $\{T_{\lambda },T_{\mu ^{\ast
}}:\lambda ,u\in E^{\ast }\}$ (see \cite{A15,AA05,AK13,AM12,AG12}). We do
not use this notation because we want to distinguish the paths in $E$ and
the elements of the algebra $A$.

\item[(ii)] Since $\Lambda $ is row-finite, $\left\vert \Lambda ^{\min
}\left( \lambda ,\mu \right) \right\vert $ is finite and the sum in (CP3)
makes sense. We also interpret the empty sum as $0$, so $\Lambda ^{\min
}\left( \lambda ,\mu \right) =\emptyset $ implies $T_{\lambda ^{\ast
}}T_{\mu }=0$.
\end{enumerate}
\end{remark}

Since (CP1-3) are the same as (KP1-3) of \cite[Definition 3.1]{CP15},
Proposition 3.3 of \cite{CP15} also applies to Cohn $\Lambda $-families as
stated in the following proposition.

\begin{proposition}
\label{properties-of-CP}Suppose that $\Lambda $ is a row-finite $k$-graph
with no sources, $R$ is a commutative ring with $1$, and $\{T_{\lambda
},T_{\mu ^{\ast }}:\lambda ,u\in \Lambda \}$ is a Cohn $\Lambda $-family in
an $R$-algebra $A$. Then

\begin{enumerate}
\item[(a)] $T_{\lambda }T_{\lambda ^{\ast }}T_{\mu }T_{\mu ^{\ast
}}=\sum_{\lambda \nu \in \operatorname{MCE}\left( \lambda ,\mu \right) }T_{\lambda
\nu }T_{\left( \lambda \nu \right) ^{\ast }}$ for $\lambda ,\mu \in \Lambda $%
; and $\left\{ T_{\lambda }T_{\lambda ^{\ast }}:\lambda \in \Lambda \right\}
$ is a commuting family.

\item[(b)] The subalgebra generated by $\left\{ T_{\lambda },T_{\mu ^{\ast
}}:\lambda ,u\in \Lambda \right\} $ is
\begin{equation*}
\operatorname{span}_{R}\{T_{\lambda }T_{\mu ^{\ast }}:\lambda ,u\in \Lambda ,s\left(
\lambda \right) =s\left( \mu \right) \}\text{.}
\end{equation*}
\end{enumerate}
\end{proposition}

Now we give an example of a Cohn $\Lambda $-family. We use this example
later to study properties of \textquotedblleft the universal Cohn $\Lambda $%
-family\textquotedblright\ (Theorem \ref{universal-CP-family}).

\begin{proposition}
\label{the-path-representation}Suppose that $\Lambda $ is a row-finite $k$%
-graph with no sources and $R$ is a commutative ring with $1$. Suppose that $%
\mathbb{F}_{R}\left( W_{\Lambda }\right) $ is the free $R$-module with basis
$W_{\Lambda }$. Then there exists a Cohn $\Lambda $-family $\left\{
T_{\lambda },T_{\mu ^{\ast }}:\lambda ,u\in \Lambda \right\} $ in the $R$%
-algebra $\operatorname{End}\left( \mathbb{F}_{R}\left( W_{\Lambda }\right) \right) $
such that for $v\in \Lambda ^{0}$, $\lambda ,u\in \Lambda $ and $x\in
W_{\Lambda }$, we have%
\begin{align*}
T_{v}\left( x\right) & =%
\begin{cases}
x & \text{if }r\left( x\right) =v\text{;} \\
0 & \text{otherwise,}%
\end{cases}
\\
T_{\lambda }\left( x\right) & =%
\begin{cases}
\lambda x & \text{if }s\left( \lambda \right) =r\left( x\right) \text{;} \\
0 & \text{otherwise,}%
\end{cases}
\\
T_{\mu ^{\ast }}\left( x\right) & =%
\begin{cases}
\sigma ^{d\left( \mu \right) }x & \text{if }x\left( 0,d\left( \mu \right)
\right) =\mu \text{;} \\
0 & \text{otherwise.}%
\end{cases}%
\end{align*}%
Furthermore, $rT_{v}\neq 0$ and $r\prod_{e\in v\Lambda ^{1}}\left(
T_{v}-T_{e}T_{e}^{\ast }\right) \neq 0$ for all $r\in \left.
R\right\backslash \left\{ 0\right\} $ and $v\in \Lambda ^{0}$.
\end{proposition}

\begin{proof}
We modify the construction of the infinite-path representation of \cite%
{ACaHR13}. Take $v\in \Lambda ^{0}$ and $\lambda ,\mu \in \left. \Lambda
\right\backslash \Lambda ^{0}$. Define functions $f_{v}$, $f_{\lambda }$,
and $f_{\mu ^{\ast }}:W_{\Lambda }\rightarrow \mathbb{F}_{R}\left(
W_{\Lambda }\right) $ by%
\begin{align*}
f_{v}\left( x\right) & =%
\begin{cases}
x & \text{if }r\left( x\right) =v\text{;} \\
0 & \text{otherwise,}%
\end{cases}
\\
f_{\lambda }\left( x\right) & =%
\begin{cases}
\lambda x & \text{if }s\left( \lambda \right) =r\left( x\right) \text{;} \\
0 & \text{otherwise,}%
\end{cases}
\\
f_{\mu ^{\ast }}\left( x\right) & =%
\begin{cases}
\sigma ^{d\left( \mu \right) }x & \text{if }x\left( 0,d\left( \mu \right)
\right) =\mu \text{;} \\
0 & \text{otherwise.}%
\end{cases}%
\end{align*}%
By the universal property of free modules, there exist nonzero endomorphisms
$S_{v},S_{\lambda },S_{\mu ^{\ast }}:\mathbb{F}_{R}\left( W_{\Lambda
}\right) \rightarrow \mathbb{F}_{R}\left( W_{\Lambda }\right) $ extending $%
f_{v}$, $f_{\lambda }$, and $f_{\mu ^{\ast }}$.

We claim that $\left\{ T_{\lambda },T_{\mu ^{\ast }}:\lambda ,u\in \Lambda
\right\} $ is a Cohn $\Lambda $-family. First we show (CP1). Take $v\in
\Lambda ^{0}$. Then $T_{v}^{2}\left( x\right) =x=T_{v}\left( x\right) $ if $%
r\left( x\right) =v$, and $T_{v}^{2}\left( x\right) =0=T_{v}\left( x\right) $
otherwise. Hence $T_{v}^{2}=T_{v}$. Now take $v,w\in \Lambda ^{0}$ with $%
v\neq w$. Then $x\in wW_{\Lambda }$ implies $x\notin vW_{\Lambda }$. Thus $%
T_{v}T_{w}\left( x\right) =0$ for every $x\in W_{\Lambda }$, and then $%
T_{v}T_{w}=0$.

Next we show (CP2). Take $\lambda ,\mu \in \Lambda $ with $s\left( \lambda
\right) =r\left( \mu \right) $. Then $T_{\lambda }T_{\mu }\left( x\right)
=\lambda \mu x=T_{\lambda \mu }\left( x\right) $ if $x\in s\left( \mu
\right) W_{\Lambda }$, and $T_{\lambda }T_{\mu }\left( x\right)
=0=T_{\lambda \mu }\left( x\right) $ otherwise. Hence $T_{\lambda }T_{\mu
}=T_{\lambda \mu }$. On the other hand, we have
\begin{equation*}
T_{\mu ^{\ast }}T_{\lambda ^{\ast }}\left( x\right) =T_{\mu ^{\ast }}\sigma
^{d\left( \lambda \right) }x=\sigma ^{d\left( \lambda \right) +d\left( \mu
\right) }x=\sigma ^{d\left( \lambda \mu \right) }x=T_{\left( \lambda \mu
\right) ^{\ast }}\left( x\right)
\end{equation*}%
if $x\left( 0,d\left( \lambda \mu \right) \right) =\lambda \mu $, and $%
T_{\mu ^{\ast }}T_{\lambda ^{\ast }}\left( x\right) =0=T_{\left( \lambda \mu
\right) ^{\ast }}\left( x\right) $ otherwise. Therefore, $T_{\mu ^{\ast
}}T_{\lambda ^{\ast }}=T_{\left( \lambda \mu \right) ^{\ast }}$.

Next we show (CP3). Take $\lambda ,\mu \in \Lambda $. If $r\left( \lambda
\right) \neq r\left( \mu \right) $, then $T_{\lambda ^{\ast }}T_{\mu }=0$
and $\Lambda ^{\min }\left( \lambda ,\mu \right) =\emptyset $, as required.
Suppose $r\left( \lambda \right) =r\left( \mu \right) $. We have%
\begin{equation*}
T_{\lambda ^{\ast }}T_{\mu }\left( x\right) =%
\begin{cases}
\left( \mu x\right) \left( d\left( \lambda \right) ,d\left( \mu x\right)
\right) & \text{if }x\in s\left( \mu \right) W_{\Lambda }\text{ and }\left(
\mu x\right) \left( 0,d\left( \lambda \right) \right) =\lambda \text{;} \\
0 & \text{otherwise.}%
\end{cases}%
\end{equation*}%
Take $x\in s\left( \mu \right) W_{\Lambda }$. Note that $s\left( \mu \right)
=r\left( \gamma \right) $ for $(\nu ,\gamma )\in \Lambda ^{\min }\left(
\lambda ,\mu \right) $. First suppose $\left( \mu x\right) \left( 0,d\left(
\lambda \right) \right) \neq \lambda $. Then for $(\nu ,\gamma )\in \Lambda
^{\min }\left( \lambda ,\mu \right) $,%
\begin{equation*}
\left( \mu x\right) \left( 0,d\left( \lambda \nu \right) \right) \neq
\lambda \nu \text{ and }\left( \mu x\right) \left( 0,d\left( \mu \gamma
\right) \right) \neq \mu \gamma \text{.}
\end{equation*}%
Hence $x\left( 0,d\left( \gamma \right) \right) \neq \gamma $ and $T_{\nu
}T_{\gamma ^{\ast }}\left( x\right) =T_{\nu }\left( 0\right) =0$. Therefore
\begin{equation*}
\sum_{(\nu ,\gamma )\in \Lambda ^{\min }\left( \lambda ,\mu \right) }T_{\nu
}T_{\gamma ^{\ast }}\left( x\right) =0.
\end{equation*}%
Next suppose $\left( \mu x\right) \left( 0,d\left( \lambda \right) \right)
=\lambda $. Since $\left( \mu x\right) \left( 0,d\left( \lambda \right)
\right) =\lambda $ and $\left( \mu x\right) \left( 0,d\left( \mu \right)
\right) =\mu $, then there is $\gamma \in s\left( \mu \right) \Lambda $ such
that $(\nu ,\gamma )\in \Lambda ^{\min }\left( \lambda ,\mu \right) $ and $%
\left( \mu x\right) \left( 0,d\left( \mu \gamma \right) \right) =\mu \gamma $%
. Therefore $x\left( 0,d\left( \gamma \right) \right) =\gamma $. Note that
this $\gamma $ is unique by the factorisation property. Hence for $(\nu
^{\prime },\gamma ^{\prime })\in \Lambda ^{\min }\left( \lambda ,\mu \right)
$ such that $(\nu ^{\prime },\gamma ^{\prime })\neq (\nu ,\gamma )$, we have
$T_{\nu ^{\prime }}T_{\gamma ^{\prime \ast }}\left( x\right) =0$. Since $%
x\left( 0,d\left( \gamma \right) \right) =\gamma $, then%
\begin{align*}
T_{\nu }T_{\gamma ^{\ast }}\left( x\right) & =T_{\nu }\left( x\left( d\left(
\gamma \right) ,d\left( x\right) \right) \right) =\nu \left[ x\left( d\left(
\gamma \right) ,d\left( x\right) \right) \right] \\
& =\nu \left[ \left( \mu x\right) \left( d\left( \mu \gamma \right) ,d\left(
\mu x\right) \right) \right] \\
& =\nu \left[ \left( \mu x\right) \left( d\left( \lambda \gamma \right)
,d\left( \mu x\right) \right) \right] \text{ (since }\mu \gamma =\lambda \nu
\text{)} \\
& =\left( \mu x\right) \left( d\left( \lambda \right) ,d\left( \mu x\right)
\right)
\end{align*}%
and%
\begin{equation*}
\sum_{(\nu ^{\prime },\gamma ;)\in \Lambda ^{\min }\left( \lambda ,\mu
\right) }T_{\nu ^{\prime }}T_{\gamma ^{\prime \ast }}\left( x\right) =T_{\nu
}T_{\gamma ^{\ast }}\left( x\right) =\left( \mu x\right) \left( d\left(
\lambda \right) ,d\left( \mu x\right) \right) =T_{\lambda ^{\ast }}T_{\mu
}\left( x\right) \text{,}
\end{equation*}%
as required. Thus $\left\{ T_{\lambda },T_{\mu ^{\ast }}:\lambda ,u\in
\Lambda \right\} $ is a Cohn $\Lambda $-family, as claimed.

Finally we show that $rT_{v}\neq 0$ and $r\prod_{e\in v\Lambda ^{1}}\left(
T_{v}-T_{e}T_{e}^{\ast }\right) \neq 0$ for all $r\in \left.
R\right\backslash \left\{ 0\right\} $ and $v\in \Lambda ^{0}$. Take $r\in
\left. R\right\backslash \left\{ 0\right\} $ and $v\in \Lambda ^{0}$. Then $%
v\in W_{\Lambda }$. Hence $rT_{v}\left( v\right) =rv$ and $rT_{v}\neq 0$. On
the other hand, for $e\in v\Lambda ^{1}$, we have $T_{e}^{\ast }\left(
v\right) =0$ and then
\begin{equation*}
r\prod_{e\in v\Lambda ^{1}}\left( T_{v}-T_{e}T_{e}^{\ast }\right) \left(
v\right) =rT_{v}\left( v\right) =rv\text{.}
\end{equation*}%
Hence, $r\prod_{e\in v\Lambda ^{1}}\left( T_{v}-T_{e}T_{e}^{\ast }\right)
\neq 0$, as required.
\end{proof}

Next we show that there is an $R$-algebra which is universal for Cohn $%
\Lambda $-families.

\begin{theorem}
\label{universal-CP-family}Suppose that $\Lambda $ is a row-finite $k$-graph
with no sources and $R$ is a commutative ring with $1$.

\begin{enumerate}
\item[(a)] There is a universal $R$-algebra ${\normalsize \operatorname{C}}%
_{R}\left( \Lambda \right) $ generated by a Cohn $\Lambda $-family $%
\{t_{\lambda },t_{\mu ^{\ast }}:\lambda ,u\in \Lambda \}$ such that if $%
\left\{ T_{\lambda },T_{\mu ^{\ast }}:\lambda ,u\in \Lambda \right\} $ is a
Cohn $\Lambda $-family in an $R$-algebra $A$, then there exists an unique $R$%
-algebra homomorphism $\phi _{T}:{\normalsize \operatorname{C}}_{R}\left( \Lambda
\right) \rightarrow A$ such that $\phi _{T}\left( t_{\lambda }\right)
=T_{\lambda }$ and $\phi _{T}\left( t_{\mu ^{\ast }}\right) =T_{\mu ^{\ast
}} $ for $\lambda ,\mu \in \Lambda $.

\item[(b)] We have $rt_{v}\neq 0$ and $r\prod_{e\in v\Lambda ^{1}}\left(
t_{v}-t_{e}t_{e}^{\ast }\right) \neq 0$ for all $r\in \left.
R\right\backslash \left\{ 0\right\} $ and $v\in \Lambda ^{0}$.
\end{enumerate}
\end{theorem}

\begin{proof}
Let $X:=\Lambda \cup G\left( \Lambda ^{\neq 0}\right) $ and $\mathbb{F}%
_{R}\left( w\left( X\right) \right) $ be the free algebra on the set $%
w\left( X\right) $ of words on $X$. Let $I$ be the ideal of $\mathbb{F}%
_{R}\left( w\left( X\right) \right) $ generated by elements of the following
sets:

\begin{enumerate}
\item[(i)] $\left\{ vw-\delta _{v,w}v:v,w\in \Lambda ^{0}\right\} $,

\item[(ii)] $\{\lambda -\mu \nu ,\lambda ^{\ast }-\nu ^{\ast }\mu ^{\ast
}:\lambda ,\mu ,\nu \in \Lambda $ and $\lambda =\mu \nu \}$ and

\item[(iii)] $\{\lambda ^{\ast }\mu -\sum_{(\nu ,\gamma )\in \Lambda ^{\min
}\left( \lambda ,\mu \right) }\nu \gamma ^{\ast }:\lambda ,\mu \in \Lambda
\} $.
\end{enumerate}

Set ${\normalsize \operatorname{C}}_{R}\left( \Lambda \right) :=\mathbb{F}_{R}\left(
w\left( X\right) \right) /I$ and write $q:\mathbb{F}_{R}\left( w\left(
X\right) \right) \rightarrow \mathbb{F}_{R}\left( w\left( X\right) \right)
/I $ for the quotient map. Define $t_{\lambda }:=q\left( \lambda \right) $
for $\lambda \in \Lambda $, and $t_{\mu ^{\ast }}:=q\left( \mu ^{\ast
}\right) $ for $\mu ^{\ast }\in G\left( \Lambda ^{\neq 0}\right) $. Then $%
\{t_{\lambda },t_{\mu ^{\ast }}:\lambda ,\mu \in \Lambda \}$ is a Cohn $%
\Lambda $-family in ${\normalsize \operatorname{C}}_{R}\left( \Lambda \right) $.

Now suppose that $\left\{ T_{\lambda },T_{\mu ^{\ast }}:\lambda ,u\in
\Lambda \right\} $ is a Cohn $\Lambda $-family in an $R$-algebra $A$. Define
$f:X\rightarrow A$ by $f\left( \lambda \right) :=T_{\lambda }$ for $\lambda
\in \Lambda $, and $f\left( \mu ^{\ast }\right) :=T_{\mu ^{\ast }}$ for $\mu
^{\ast }\in G\left( \Lambda ^{\neq 0}\right) $. By the universal property of
$\mathbb{F}_{R}\left( w\left( X\right) \right) $, there exists an unique $R$%
-algebra homomorphism $\pi :\mathbb{F}_{R}\left( w\left( X\right) \right)
\rightarrow A$ such that $\pi |_{X}=f$. Since $\left\{ T_{\lambda },T_{\mu
^{\ast }}:\lambda ,u\in \Lambda \right\} $ is a Cohn $\Lambda $-family, then
$I\subseteq \ker \left( \pi \right) $. Thus there exists an $R$-algebra
homomorphism $\phi _{T}:{\normalsize \operatorname{C}}_{R}\left( \Lambda \right)
\rightarrow A$ such that $\phi _{T}\circ q=\pi $. The homomorphism $\phi
_{T} $ is unique since the element in $X$ generate $\mathbb{F}_{R}\left(
w\left( X\right) \right) $ as an algebra. Furthermore, we have $\phi
_{T}\left( t_{\lambda }\right) =T_{\lambda }$ and $\phi _{T}\left( t_{\mu
^{\ast }}\right) =T_{\mu ^{\ast }}$ for $\lambda ,\mu \in \Lambda $, as
required.

For (b), suppose that $\left\{ T_{\lambda },T_{\mu ^{\ast }}:\lambda ,u\in
\Lambda \right\} $ is the Cohn $\Lambda $-family as in Proposition \ref%
{the-path-representation}. Then $rT_{v}\neq 0$ and
$r\prod_{e\in v\Lambda ^{1}}\left( T_{v}-T_{e}T_{e}^{\ast }\right)\neq 0$ for all $r\in \left. R\right\backslash
\left\{ 0\right\} $ and $v\in \Lambda ^{0}$. Since $\phi _{T}\left(
rt_{v}\right) =rT_{v}\neq 0$ and
\begin{equation*}
\phi \Big(r\prod_{e\in v\Lambda ^{1}}\left( t_{v}-t_{e}t_{e}^{\ast }\right) %
\Big)=r\prod_{e\in v\Lambda ^{1}}\left( T_{v}-T_{e}T_{e}^{\ast }\right) \neq
0
\end{equation*}%
for all $r\in \left. R\right\backslash \left\{ 0\right\} $ and $v\in \Lambda
^{0}$, then $rt_{v}\neq 0$ and $r\prod_{e\in v\Lambda ^{1}}\left(
t_{v}-t_{e}t_{e}^{\ast }\right) \neq 0$ for all $r\in \left.
R\right\backslash \left\{ 0\right\} $ and $v\in \Lambda ^{0}$.
\end{proof}

\section{The uniqueness theorem for Cohn path algebras}

\label{Section-the-uniqueness-theorem-of-CP-family}In this section, we
establish a uniqueness theorem for Cohn path algebras (Theorem \ref%
{the-uniqueness-theorem-of-CP-family}). This uniqueness theorem can be
viewed as an algebraic analogue of the uniqueness theorem for Toeplitz
algebras \cite[Theorem 8.1]{RS05}. Note that in \cite{RS05}, Raeburn and
Sims state the uniqueness theorem for Toeplitz algebras in terms of the more
general product systems of graphs over $\mathbb{N}^{k}$; thus their result
includes Toeplitz algebras associated to $k$-graphs. In \cite[Theorem 2.2]%
{P15}, Pangalela states the uniqueness theorem in the $k$-graph setting
explicitly. For further discussion, see Remark 2.3 and Remark 2.4 of \cite%
{P15}. This uniqueness theorem also does not require any hypothesis on the $%
k $-graph and thus applies generally.

\begin{theorem}[The uniqueness theorem for Cohn path algebras]
\label{the-uniqueness-theorem-of-CP-family}Let $\Lambda $ be a row-finite $k$%
-graph with no sources and $R$ be a commutative ring with $1$. Suppose that $%
\phi :{\normalsize \operatorname{C}}_{R}\left( \Lambda \right) \rightarrow A$ is %
a ring homomorphism such that
\begin{equation*}
\phi \left( rt_{v}\right) \neq 0\text{ and }\phi \Big(r\prod_{e\in v\Lambda
^{1}}\left( t_{v}-t_{e}t_{e^{\ast }}\right) \Big)\neq 0
\end{equation*}%
for all $r\in \left. R\right\backslash \left\{ 0\right\} $ and $v\in \Lambda
^{0}$. Then $\phi $ is injective.
\end{theorem}

The rest of this section is devoted to proving Theorem \ref%
{the-uniqueness-theorem-of-CP-family}. To help readers follow our proofs, we
divide the arguments into three subsections. In Subsection \ref%
{Subsection-KP-algebras}, we recall the Kumjian-Pask $\Lambda$-families of
\cite{CP15} and study some of their properties. In Subsection \ref%
{Subsection-k-graph-TLambda}, we recall the $k$-graph $T\Lambda $ of \cite%
{P15} and investigate the Kumjian-Pask algebra of $T\Lambda $. Finally, in
Subsection \ref{Subsection-Cohn-path-algebras-and-KP-algebras}, we show that
every Cohn $\Lambda $-family is isomorphic to the Kumjian-Pask $T\Lambda $%
-family (Theorem \ref{CP-is-isomorphic-to-KP}). Once we have this
isomorphism, we show that Theorem \ref{the-uniqueness-theorem-of-CP-family}
is a consequence of the Cuntz-Krieger uniqueness theorem for Kumjian-Pask
algebras \cite[Theorem 8.1]{CP15}.

\subsection{Kumjian-Pask algebras}

\label{Subsection-KP-algebras}Suppose that $\Lambda $ is a row-finite $k$%
-graph. Recall from \cite[Definition 3.1]{CP15} that a \emph{Kumjian-Pask }$%
\Lambda $\emph{-family} $\left\{ S_{\lambda },S_{\mu ^{\ast }}:\lambda ,u\in
\Lambda \right\} $ in an $R$-algebra $A$ is a family which satisfies (CP1-3)
and

\begin{enumerate}
\item[(KP)] $\prod_{\lambda \in E}\left( S_{v}-S_{\lambda }S_{\lambda ^{\ast
}}\right) =0$ for all $v\in \Lambda ^{0}$ and finite exhaustive $E\subseteq
v\Lambda .$
\end{enumerate}

\begin{remark}
We are careful to not say that a Kumjian-Pask $\Lambda $-family is a Cohn $%
\Lambda $-family which satisfies (KP). This is because in Definition \ref%
{CP-family}, we define Cohn $\Lambda $-family of row-finite $k$-graphs with
no sources; however, the above definition of Kumjian-Pask $\Lambda $-family
allows for more general row-finite $k$-graphs (in particular, to $k$-graphs
with sources). We will need this level of generality later on.
\end{remark}

For a row-finite $k$-graph $\Lambda $, there exists an $R$-algebra $%
{\normalsize \operatorname{KP}}_{R}\left( \Lambda \right) $ generated by the
universal Kumjian-Pask $\Lambda $-family $\left\{ s_{\lambda },s_{\mu ^{\ast
}}:\lambda ,u\in \Lambda \right\} $.

\begin{remark}
For a row-finite $k$-graph $\Lambda $ with no sources, the set $v\Lambda
^{e_{i}}$ is exhaustive for all $v\in \Lambda ^{0}$ and $1\leq i\leq k$ (see
\cite[Proof of Lemma B.2]{RSY04}). Hence, ${\normalsize \operatorname{KP}}_{R}\left(
\Lambda \right) $ is a nontrivial quotient of ${\normalsize \operatorname{C}}%
_{R}\left( \Lambda \right) $ and then ${\normalsize \operatorname{C}}_{R}\left(
\Lambda \right) $ is not simple.
\end{remark}

There is a powerful uniqueness theorem for Kumjian-Pask algebras, called the
Cuntz-Krieger uniqueness theorem which we prove Theorem 8.1 of \cite{CP15}.
Since we shall use this theorem in the proof of Theorem \ref%
{the-uniqueness-theorem-of-CP-family} (see Subsection \ref%
{Subsection-Cohn-path-algebras-and-KP-algebras}), we state it explicitly:

\begin{theorem}
\label{the-CK-uniqueness-theorem}Let $\Lambda $ be a row-finite $k$-graph
which satisfies the \emph{aperiodicity condition}:
\begin{align*}
\text{for every pair of distinct paths }\lambda ,\mu & \in \Lambda \text{
with }s\left( \lambda \right) =s\left( \mu \right) \text{,} \\
\text{there exists }\eta & \in s\left( \lambda \right) \Lambda \text{ such
that }\operatorname{MCE}\left( \lambda \eta ,\mu \eta \right) =\emptyset \text{.}
\end{align*}%
Let $R$ be a commutative ring with $1$. Suppose that $\pi :{\normalsize
\operatorname{KP}}_{R}\left( \Lambda \right) \rightarrow A$ is a ring homomorphism
such that $\pi \left( rs_{v}\right) \neq 0$ for all $r\in \left.
R\right\backslash \left\{ 0\right\} $ and $v\in \Lambda ^{0}$. Then $\pi $
is injective.
\end{theorem}

\begin{remark}
Our aperiodicity condition is from \cite[Definition 3.1]{LS10} and applies
to very general (not necessarily row-finite) $k$-graphs. In the setting of
row-finite higher-rank graphs, our definition is equivalent to various
aperiodicity definitions, including Condition (B$^{\prime }$) of \cite[%
Remark 7.3]{FMY05} and \cite[Definition 2.1.(ii)]{Sh12} (see \cite{LS10,
RS07, RS09,Sh12}).
\end{remark}

\begin{remark}
The aperiodicity condition is a higher-rank analogue of Condition (L) for $1
$-graphs. Using our path convention, Condition (L)\ says that every cycle
has an entry (see \cite{A15,Leavitt path algebras,CBMS}).
\end{remark}

The following proposition will be useful to simplify calculations in
Kumjian-Pask algebras. In essence gives an alternate formulation of (KP).

\begin{proposition}
\label{KP-family-on-edges-KP}Let $\Lambda $ be a row-finite $k$-graph and $R$
be a commutative ring with $1$. Suppose that $\left\{ S_{\lambda },S_{\mu
^{\ast }}:\lambda ,u\in \Lambda \right\} $ is a Cohn $\Lambda $-family in an
$R$-algebra $A$. Then
\begin{equation*}
\left\{ S_{\lambda },S_{\mu ^{\ast }}:\lambda ,u\in \Lambda
\right\}
\end{equation*}
is a Kumjian-Pask $\Lambda $-family if and only if
\begin{equation*}
\prod_{e\in E}\left( S_{v}-S_{e}S_{e^{\ast }}\right) =0
\end{equation*}%
for all $v\in \Lambda ^{0}$ and exhaustive $E\subseteq v\Lambda ^{1}.$
\end{proposition}

Before proving Proposition \ref{KP-family-on-edges-KP}, we establish the
following helper lemma.

\begin{lemma}
\label{KP-family-on-edges-stepstone-lemma}Let $\Lambda $ be a row-finite $k$%
-graph and $R$ be a commutative ring with $1$. Suppose that
$\left\{ S_{\lambda },S_{\mu ^{\ast }}:\lambda ,u\in \Lambda \right\} $is a
Cohn $\Lambda $-family in an $R$-algebra $A$. Suppose $v\in \Lambda ^{0}$, $%
\lambda \in v\Lambda $ and $E\subseteq s\left( \lambda \right) \Lambda $ is
finite and satisfies $\prod_{\nu \in E}\left( S_{s\left( \lambda \right)
}-S_{\nu }S_{\nu ^{\ast }}\right) =0$. Then%
\begin{equation*}
S_{v}-S_{\lambda }S_{\lambda ^{\ast }}=\prod_{\nu \in E}\left(
S_{v}-S_{\lambda \nu }S_{\left( \lambda \nu \right) ^{\ast }}\right) \text{.}
\end{equation*}
\end{lemma}

\begin{proof}
We follow the $C^{\ast }$-algebraic argument of \cite[Lemma C.7]{RSY04}. For
$\nu \in s\left( \lambda \right) \Lambda $, we have%
\begin{equation*}
\left( S_{v}-S_{\lambda }S_{\lambda ^{\ast }}\right) \left( S_{v}-S_{\lambda
\nu }S_{\left( \lambda \nu \right) ^{\ast }}\right) =S_{v}-S_{\lambda
}S_{\lambda ^{\ast }}\text{;}
\end{equation*}%
so%
\begin{equation*}
\left( S_{v}-S_{\lambda }S_{\lambda ^{\ast }}\right) \prod_{\nu \in E}\left(
S_{v}-S_{\lambda \nu }S_{\left( \lambda \nu \right) ^{\ast }}\right)
=S_{v}-S_{\lambda }S_{\lambda ^{\ast }}.
\end{equation*}%
On the other hand,%
\begin{align*}
\left( S_{v}-S_{\lambda }S_{\lambda ^{\ast }}\right) \prod_{\nu \in E}\left(
S_{v}-S_{\lambda \nu }S_{\left( \lambda \nu \right) ^{\ast }}\right) &
=S_{v}\prod_{\nu \in E}\left( S_{v}-S_{\lambda \nu }S_{\left( \lambda \nu
\right) ^{\ast }}\right) -S_{\lambda }S_{\lambda ^{\ast }}\prod_{\nu \in
E}\left( S_{v}-S_{\lambda \nu }S_{\left( \lambda \nu \right) ^{\ast }}\right)
\\
& =\prod_{\nu \in E}\left( S_{v}-S_{\lambda \nu }S_{\left( \lambda \nu
\right) ^{\ast }}\right) -\prod_{\nu \in E}\left( S_{\lambda }S_{\lambda
^{\ast }}-S_{\lambda \nu }S_{\left( \lambda \nu \right) ^{\ast }}\right) \\
& =\prod_{\nu \in E}\left( S_{v}-S_{\lambda \nu }S_{\left( \lambda \nu
\right) ^{\ast }}\right) -S_{\lambda }\big(\prod_{\nu \in E}\left(
S_{s\left( \lambda \right) }-S_{\nu }S_{\nu ^{\ast }}\right) \big)S_{\lambda
^{\ast }} \\
& =\prod_{\nu \in E}\left( S_{v}-S_{\lambda \nu }S_{\left( \lambda \nu
\right) ^{\ast }}\right)
\end{align*}%
since $\prod_{\nu \in E}\left( S_{s\left( \lambda \right) }-S_{\nu }S_{\nu
^{\ast }}\right) =0$ by the hypothesis. The conclusion follows.
\end{proof}

\begin{proof}[Proof of Proposition \protect\ref{KP-family-on-edges-KP}]
We use a similar argument to the $C^{\ast }$-algebraic version in \cite[%
Proposition C.3]{RSY04}. If $\{S_{\lambda },S_{\mu ^{\ast }}:\lambda ,u\in
\Lambda \}$ is a Kumjian-Pask $\Lambda $-family, then it satisfies $%
\prod_{e\in E}\left( S_{v}-S_{e}S_{e^{\ast }}\right) =0$ for all $v\in
\Lambda ^{0}$ and exhaustive set $E\subseteq v\Lambda ^{1}.$ Now we show the
reverse implication. First for $E\subseteq \Lambda $, we write%
\begin{equation*}
I\left( E\right) :=\bigcup_{i=1}^{k}\left\{ \lambda \left( 0,e_{i}\right)
:\lambda \in E,d\left( \lambda \right) _{i}>0\right\} \text{ and }L\left(
E\right) :=\sum_{i=1}^{k}\max_{\lambda \in E}d\left( \lambda \right) _{i}%
\text{.}
\end{equation*}%
We have to show $\prod_{\lambda \in E}\left( S_{v}-S_{\lambda }S_{\lambda
^{\ast }}\right) =0$ for all $v\in \Lambda ^{0}$ and exhaustive set $%
E\subseteq v\Lambda $. We show this by induction on $L\left( E\right) $. If $%
L\left( E\right) =1$, then $E\subseteq v\Lambda ^{1}$ for some $v\in \Lambda
^{0}$ and $\prod_{e\in E}\left( S_{v}-S_{e}S_{e^{\ast }}\right) =0$ by
assumption.

Now fix $l\geq 1$ and suppose that $\prod_{\lambda \in F}\left(
S_{v}-S_{\lambda }S_{\lambda ^{\ast }}\right) =0$ for all $v\in \Lambda ^{0}$
and exhaustive set $F\subseteq v\Lambda $ with $L\left( F\right) \leq l$.
Take $v\in \Lambda ^{0}$ and exhaustive set $E\subseteq v\Lambda $ with $%
L\left( E\right) =l+1$. If $v\in E$, then $\prod_{\lambda \in E}\left(
S_{v}-S_{\lambda }S_{\lambda ^{\ast }}\right) =0$. So suppose $v\notin E$.
Note that $I\left( E\right) \subseteq v\Lambda ^{1}$. Since $E$ is
exhaustive, then by \cite[Lemma C.6]{RSY04}, $I\left( E\right) $ is also
exhaustive. So%
\begin{equation}
I\left( E\right) \subseteq v\Lambda ^{1}\text{ is exhaustive}.
\label{equ-I(E)-is-exhaustive}
\end{equation}%
Take $e\in I\left( E\right) $ and by \cite[Lemma C.5]{RSY04}, $\operatorname{Ext}%
\left( e;E\right) $ is exhaustive. By \cite[Lemma C.8]{RSY04}, $L\left(
\operatorname{Ext}\left( e;E\right) \right) <L\left( E\right) =l+1$ and then $%
L\left( \operatorname{Ext}\left( e;E\right) \right) \leq l$. So by the inductive
hypothesis, $\prod_{\nu \in \operatorname{Ext}\left( e;E\right) }\left( S_{s\left(
e\right) }-S_{\nu }S_{\nu ^{\ast }}\right) =0$ and then by Lemma \ref%
{KP-family-on-edges-stepstone-lemma}, we get%
\begin{equation}
S_{v}-S_{e}S_{e^{\ast }}=\prod_{\nu \in \operatorname{Ext}\left( e;E\right) }\left(
S_{v}-S_{e\nu }S_{\left( e\nu \right) ^{\ast }}\right) \text{.}
\label{equ-KP-family-on-edges}
\end{equation}%
Now note that for $\nu \in \operatorname{Ext}\left( e;E\right) $, there exists $%
\lambda \in E$ with $e\nu =\lambda \lambda ^{\prime }$, and then
\begin{equation*}
\left( S_{v}-S_{\lambda }S_{\lambda ^{\ast }}\right) \left( S_{v}-S_{e\nu
}S_{\left( e\nu \right) ^{\ast }}\right) =S_{v}-S_{\lambda }S_{\lambda
^{\ast }}\text{.}
\end{equation*}%
Hence%
\begin{align*}
\prod_{\lambda \in E}\left( S_{v}-S_{\lambda }S_{\lambda ^{\ast }}\right) & =%
\big(\prod_{\lambda \in E}\left( S_{v}-S_{\lambda }S_{\lambda ^{\ast
}}\right) \big)\big(\prod_{e\in I\left( E\right) }\prod_{\nu \in \operatorname{Ext}%
\left( e;E\right) }\left( S_{v}-S_{e\nu }S_{\left( e\nu \right) ^{\ast
}}\right) \big) \\
& =\big(\prod_{\lambda \in E}\left( S_{v}-S_{\lambda }S_{\lambda ^{\ast
}}\right) \big)\big(\prod_{e\in I\left( E\right) }\left(
S_{v}-S_{e}S_{e^{\ast }}\right) \big)\text{ (by %
\eqref{equ-KP-family-on-edges})} \\
& =\big(\prod_{\lambda \in E}\left( S_{v}-S_{\lambda }S_{\lambda ^{\ast
}}\right) \big)\left( 0\right) \text{ (by \eqref{equ-I(E)-is-exhaustive} and
the inductive hypothesis)} \\
& =0\text{,}
\end{align*}%
as required.
\end{proof}

\subsection{The $k$-graph $T\Lambda $ and Kumjian-Pask $T\Lambda $-families}

\label{Subsection-k-graph-TLambda}As in \cite[Proposition 3.1]{P15}, for a
row-finite $k$-graph with no sources $\Lambda $, $T\Lambda $ is the $k$%
-graph where%
\begin{equation*}
T\Lambda ^{0}:=\left\{ \alpha \left( v\right) ,\beta \left( v\right) :v\in
\Lambda ^{0}\right\} \text{, }T\Lambda :=\left\{ \alpha \left( \lambda
\right) ,\beta \left( \lambda \right) :\lambda \in \Lambda \right\} \,\text{,%
}
\end{equation*}%
and for $\alpha \left( \lambda \right) ,\beta \left( \lambda \right) \in
\left. T\Lambda \right\backslash T\Lambda ^{0}$,%
\begin{align*}
r\left( \alpha \left( \lambda \right) \right) :=\alpha \left( r\left(
\lambda \right) \right) & \text{, }s\left( \alpha \left( \lambda \right)
\right) :=\alpha \left( s\left( \lambda \right) \right) \text{,} \\
r\left( \beta \left( \lambda \right) \right) :=\alpha \left( r\left( \lambda
\right) \right) & \text{, }s\left( \beta \left( \lambda \right) \right)
:=\beta \left( s\left( \lambda \right) \right) \text{.}
\end{align*}%
Note that every vertex $\beta \left( v\right) $ satisfies $\beta \left(
v\right) T\Lambda =\left\{ \beta \left( v\right) \right\} $ and by \cite[%
Proposition 3.4]{P15}, $T\Lambda $ is row-finite and aperiodic. The next
proposition characterises exhaustive sets of $T\Lambda $.

\begin{proposition}
\label{characterisation-of-tlambda-when-lambda-has-no-sources}Suppose that $%
\Lambda $ is a row-finite $k$-graph with no sources. Then for every $\alpha
\left( v\right) \in T\Lambda ^{0}$, the only exhaustive set contained in $%
\alpha \left( v\right) T\Lambda ^{1}$ is $\alpha \left( v\right) T\Lambda
^{1}$ itself.
\end{proposition}

\begin{proof}
Fix an exhaustive set $E\subseteq \alpha \left( v\right) T\Lambda ^{1}$. We
have to show $E=\alpha \left( v\right) T\Lambda ^{1}$. Since $E$ is
exhaustive, for $\beta \left( e\right) \in \alpha \left( v\right) T\Lambda
^{1}$, there exists an edge $\tau _{e}\in E$ such that $T\Lambda ^{\min
}\left( \beta \left( e\right) ,\tau _{e}\right) \neq \emptyset $. Since $%
s\left( \beta \left( e\right) \right) T\Lambda =\left\{ s\left( \beta \left(
e\right) \right) \right\} $, then $\operatorname{MCE}\left( \beta \left( e\right)
,\tau _{e}\right) =\left\{ \beta \left( e\right) \right\} $. Hence, $\tau
_{e}=\beta \left( e\right) $ because both $\tau _{e}$ and $\beta \left(
e\right) $ are edges. Thus $\beta \left( e\right) \in E$ and $E$ contains $%
\beta \left( v\Lambda ^{1}\right) $.

Now we claim $\alpha \left( v\Lambda ^{1}\right) \subseteq E$. Suppose for a
contradiction that there exist $1\leq i\leq k$ and $e\in v\Lambda ^{e_{i}}$
such that $\alpha \left( e\right) \notin E$. Since $\Lambda $ has no
sources, there exists an edge $f\in s\left( e\right) \Lambda ^{e_{i}}$. Now
consider the path $\tau =\alpha \left( e\right) \beta \left( f\right) $.
This is a path with degree $2e_{i}$ whose range at $\alpha \left( v\right) $
and $s\left( \tau \right) T\Lambda =\left\{ \beta \left( s\left( f\right)
\right) \right\} $. Since $E$ is exhaustive, there exists $\omega \in E$
such that $T\Lambda ^{\min }\left( \tau ,\omega \right) \neq \emptyset $.
Since $\tau $ is a path with length $2e_{i}$ and $s\left( \tau \right)
T\Lambda =\left\{ \beta \left( s\left( f\right) \right) \right\} $, then $%
\omega $ is either equal to $\tau $ or $\alpha \left( e\right) $. Since $%
\alpha \left( e\right) \notin E$, then $\tau =\omega \in E$, which
contradicts that $E$ only contains edges. The conclusion follows.
\end{proof}

A consequence of Proposition \ref%
{characterisation-of-tlambda-when-lambda-has-no-sources} is the following:

\begin{lemma}
\label{KP-TLambda-when-lambda-has-no-sources}Let $\Lambda $ be a row-finite $%
k$-graph with no sources and $R$ be a commutative ring with $1$. Suppose
that $\{S_{\tau },S_{\omega ^{\ast }}:\tau ,\omega \in T\Lambda \}$ is a
Cohn $T\Lambda $-family in an $R$-algebra $A$. Then the collection is a
Kumjian-Pask $T\Lambda $-family if and only if for every $\alpha \left(
v\right) \in T\Lambda ^{0}$,%
\begin{equation*}
\prod_{g\in \alpha \left( v\right) T\Lambda ^{1}}\left( S_{\alpha \left(
v\right) }-S_{g}S_{g^{\ast }}\right) =0\text{.}
\end{equation*}
\end{lemma}

\begin{proof}
If $x=\beta \left( v\right) $, then $\beta \left( v\right) T\Lambda =\left\{
\beta \left( v\right) \right\} $ and there is no exhaustive set contained in
$xT\Lambda ^{1}$. On the other hand, if $x=\alpha \left( v\right) $, by
Proposition \ref{characterisation-of-tlambda-when-lambda-has-no-sources},
the only exhaustive set contained in $\alpha \left( v\right) T\Lambda ^{1}$
is $\alpha \left( v\right) T\Lambda ^{1}$. Therefore, by Proposition \ref%
{KP-family-on-edges-KP}, $\left\{ S_{\tau },S_{\omega ^{\ast }}:\tau ,\omega
\in T\Lambda \right\} $ is a Kumjian-Pask $T\Lambda $-family if and only if $%
\prod_{g\in \alpha \left( v\right) T\Lambda ^{1}}\left( S_{\alpha \left(
v\right) }-S_{g}S_{g^{\ast }}\right) =0$ for all $\alpha \left( v\right) \in
T\Lambda ^{0}$, as required.
\end{proof}

\subsection{Relationship between Cohn $\Lambda $-families and Kumjian-Pask $%
T\Lambda $-families}

\label{Subsection-Cohn-path-algebras-and-KP-algebras}In this section, we
start out by investigating the relationship between Cohn $\Lambda $-families
and Kumjian-Pask $T\Lambda $-families (Theorem \ref{CP-is-isomorphic-to-KP}%
). Once we have this, we are then ready to prove Theorem \ref%
{the-uniqueness-theorem-of-CP-family}.

First we establish some stepping stone results (Lemma \ref{F_t,v} and Lemma %
\ref{S-is-a-KP-family}).

\begin{lemma}
\label{F_t,v}Suppose that $\left\{ T_{\lambda },T_{\mu ^{\ast }}:\lambda
,u\in \Lambda \right\} $ is a Cohn $\Lambda $-family in an $R$-algebra $A$.
For $v\in \Lambda ^{0}$, define%
\begin{equation*}
F_{T,v}:=T_{v}-\prod\limits_{e\in v\Lambda ^{1}}\left( T_{v}-T_{e}T_{e^{\ast
}}\right) \text{.}
\end{equation*}%
Then

\begin{enumerate}
\item[(a)] For $v\in \Lambda ^{0}$, we have%
\begin{equation*}
F_{T,v}=F_{T,v}^{2}\text{ and }T_{v}-F_{T,v}=\left( T_{v}-F_{T,v}\right) ^{2}%
\text{.}
\end{equation*}

\item[(b)] For every $v,w\in \Lambda ^{0}$ with $v\neq w$, we have%
\begin{equation*}
F_{T,w}F_{T,v}=0=F_{T,v}F_{T,w}\text{ and }T_{w}F_{T,v}=0=F_{T,v}T_{w}\text{.%
}
\end{equation*}

\item[(c)] For $v\in \Lambda ^{0}$ and $\lambda \in v\Lambda \backslash
\left\{ v\right\} $, we have%
\begin{equation*}
T_{v}F_{T,v}=F_{T,v}=F_{T,v}T_{v}\text{,}
\end{equation*}%
\begin{equation*}
F_{T,v}T_{\lambda }=T_{\lambda }\text{ and }T_{\lambda ^{\ast
}}F_{T,v}=T_{\lambda ^{\ast }}\text{.}
\end{equation*}

\item[(d)] Furthermore, $F_{T,v}\neq 0$ for all $r\in \left.
R\right\backslash \left\{ 0\right\} $ and $v\in \Lambda ^{0}$ if and only if
$T_{v}\neq 0$ for all $r\in \left. R\right\backslash \left\{ 0\right\} $ and
$v\in \Lambda ^{0}$.
\end{enumerate}
\end{lemma}

\begin{proof}
First we show (a). Take $v\in \Lambda ^{0}$. Note that $\left(
T_{v}-T_{e}T_{e^{\ast }}\right) ^{2}=\left( T_{v}-T_{e}T_{e^{\ast }}\right) $
for $e\in v\Lambda ^{1}$. Hence
\begin{equation*}
\left( T_{v}-F_{T,v}\right) ^{2}=\prod\limits_{e\in v\Lambda ^{1}}\left(
T_{v}-T_{e}T_{e^{\ast }}\right) ^{2}=\prod\limits_{e\in v\Lambda ^{1}}\left(
T_{v}-T_{e}T_{e^{\ast }}\right) =T_{v}-F_{T,v}
\end{equation*}%
and
\begin{equation*}
F_{T,v}^{2}=\Big(T_{v}-\prod\limits_{e\in v\Lambda ^{1}}\left(
T_{v}-T_{e}T_{e^{\ast }}\right) \Big)^{2}=T_{v}-\prod\limits_{e\in v\Lambda
^{1}}\left( T_{v}-T_{e}T_{e^{\ast }}\right) =F_{T,v}\text{.}
\end{equation*}

To show (b), we take $v,w\in \Lambda ^{0}$ with $v\neq w$. Then $%
T_{w}T_{v}=0 $ and $T_{w}T_{e}=0$ for all $e\in v\Lambda ^{1}$. Hence,%
\begin{equation*}
T_{w}F_{T,v}=T_{w}\Big(T_{v}-\prod\limits_{e\in v\Lambda ^{1}}\left(
T_{v}-T_{e}T_{e^{\ast }}\right) \Big)=0
\end{equation*}%
and by using a similar argument, we also get $F_{T,v}T_{w}=0$, as required.
On the other hand, we also have%
\begin{equation*}
F_{T,w}F_{T,v}=\Big(T_{w}-\prod\limits_{f\in w\Lambda ^{1}}\left(
T_{w}-T_{f}T_{f^{\ast }}\right) \Big)\Big(T_{v}-\prod\limits_{e\in v\Lambda
^{1}}\left( T_{v}-T_{e}T_{e^{\ast }}\right) \Big)=0
\end{equation*}%
and a similar argument also applies to get $F_{T,v}F_{T,w}=0$.

Next we show (c). We take $v\in \Lambda ^{0}$. Then%
\begin{equation*}
T_{v}F_{T,v}=T_{v}\Big(T_{v}-\prod\limits_{e\in v\Lambda ^{1}}\left(
T_{v}-T_{e}T_{e^{\ast }}\right) \Big)=T_{v}-\prod\limits_{e\in v\Lambda
^{1}}\left( T_{v}-T_{e}T_{e^{\ast }}\right) =F_{T,v}
\end{equation*}%
and since $T_{v}=T_{v^{\ast }}$, then by using a similar argument, we also
get $F_{T,v}T_{v}=F_{T,v}$.

Now take $\lambda \in v\Lambda \backslash \left\{ v\right\} $. Then there
exists $f\in v\Lambda ^{1}$ such that $ff^{\prime }=\lambda $. This implies $%
T_{f^{\ast }}T_{\lambda }=T_{f^{\prime }}$ and%
\begin{equation*}
\left( T_{v}-T_{f}T_{f^{\ast }}\right) T_{\lambda }=T_{\lambda
}-T_{f}T_{f^{\ast }}T_{\lambda }=T_{\lambda }-T_{f}T_{f^{\prime
}}=T_{\lambda }-T_{\lambda }=0\text{.}
\end{equation*}%
Hence,%
\begin{equation*}
\Big(\prod\limits_{e\in v\Lambda ^{1}}\left( T_{v}-T_{e}T_{e^{\ast }}\right) %
\Big)T_{\lambda }=0
\end{equation*}%
and%
\begin{equation*}
F_{T,v}T_{\lambda }=\Big(T_{v}-\prod\limits_{e\in v\Lambda ^{1}}\left(
T_{v}-T_{e}T_{e^{\ast }}\right) \Big)T_{\lambda }=T_{\lambda }\text{.}
\end{equation*}%
By using a similar argument, we also get $T_{\lambda ^{\ast
}}F_{T,v}=T_{\lambda ^{\ast }}$.

Finally, we show (d). First suppose that there exists $r\in \left.
R\right\backslash \left\{ 0\right\} $ and $v\in \Lambda ^{0}$ with $rT_{v}=0$%
. Then $rT_{e}=rT_{v}T_{e}=0$ for all $e\in v\Lambda ^{1}$ and then $%
rF_{T,v}=rT_{v}-r\prod_{e\in v\Lambda ^{1}}\left( T_{v}-T_{e}T_{e^{\ast
}}\right) =0$.

To show the reverse implication, suppose that there exists $r\in \left.
R\right\backslash \left\{ 0\right\} $ and $v\in \Lambda ^{0}$ with $%
rF_{T,v}=0$. Take $f\in v\Lambda ^{1}$, then%
\begin{equation}
T_{f}T_{f^{\ast }}\left( T_{v}-T_{f}T_{f^{\ast }}\right) =T_{f}T_{f^{\ast
}}-T_{f}\left( T_{f^{\ast }}T_{f}\right) T_{f^{\ast }}=T_{f}T_{f^{\ast
}}-T_{f}T_{f^{\ast }}=0\text{.}  \label{equ-TfTf(Tv-TfTf)}
\end{equation}%
Hence,%
\begin{equation*}
rT_{f}T_{f^{\ast }}=rT_{f}T_{f^{\ast }}T_{v}=rT_{f}T_{f^{\ast }}\Big(%
F_{T,v}+\prod\limits_{e\in v\Lambda ^{1}}\left( T_{v}-T_{e}T_{e^{\ast
}}\right) \Big)=rT_{f}T_{f^{\ast }}\prod\limits_{e\in v\Lambda ^{1}}\left(
T_{v}-T_{e}T_{e^{\ast }}\right) =0
\end{equation*}%
since $f\in v\Lambda ^{1}$ and \eqref{equ-TfTf(Tv-TfTf)}. Therefore,%
\begin{equation*}
rT_{f}=rT_{f}T_{s\left( f\right) }=rT_{f}\left( T_{f^{\ast }}T_{f}\right)
=\left( rT_{f}T_{f^{\ast }}\right) T_{f}=\left( 0\right) T_{f}=0
\end{equation*}%
and then%
\begin{equation*}
rT_{s\left( f\right) }=rT_{f^{\ast }}T_{f}=T_{f^{\ast }}\left( rT_{f}\right)
=T_{f^{\ast }}\left( 0\right) =0\text{,}
\end{equation*}%
as required.
\end{proof}

\begin{lemma}
\label{S-is-a-KP-family}Let $\Lambda $ be a row-finite $k$-graph with no
sources and $R$ be a commutative ring with $1$. Suppose that $\{T_{\lambda
},T_{\mu ^{\ast }}:\lambda ,u\in \Lambda \}$ is a Cohn $\Lambda $-family in
an $R$-algebra $A$. For $\tau ,\omega \in T\Lambda $, define%
\begin{align*}
S_{\tau }& :=\left\{
\begin{array}{ccc}
T_{\lambda }F_{T,s\left( \lambda \right) } & \text{if} & \tau =\alpha \left(
\lambda \right) \text{;} \\
T_{\lambda }\left( T_{s\left( \lambda \right) }-F_{T,s\left( \lambda \right)
}\right) & \text{if} & \tau =\beta \left( \lambda \right) \text{,}%
\end{array}%
\right. \\
S_{\omega ^{\ast }}& :=\left\{
\begin{array}{ccc}
F_{T,s\left( \mu \right) }T_{\mu ^{\ast }} & \text{if} & \omega =\alpha
\left( \mu \right) \text{;} \\
\left( T_{s\left( \mu \right) }-F_{T,s\left( \mu \right) }\right) T_{\mu
^{\ast }} & \text{if} & \omega =\beta \left( \mu \right) \text{.}%
\end{array}%
\right. \text{ }
\end{align*}%
Then

\begin{enumerate}
\item[(a)] $\left\{ S_{\tau },S_{\omega ^{\ast }}:\tau ,\omega \in T\Lambda
\right\} $ is a Kumjian-Pask $T\Lambda $-family.

\item[(b)] Suppose that $rT_{v}\neq 0$ and $r\prod_{e\in v\Lambda
^{1}}\left( T_{v}-T_{e}T_{e}^{\ast }\right) \neq 0$ for all $r\in \left.
R\right\backslash \left\{ 0\right\} $ and $v\in \Lambda ^{0}$. Suppose that $%
\pi _{S}:{\normalsize \operatorname{KP}}_{R}\left( T\Lambda \right) \rightarrow
{\normalsize A}$ is the $R$-algebra homomorphism such that $\pi _{S}\left(
s_{\tau }\right) =S_{\tau }$ and $\pi _{S}\left( s_{\omega ^{\ast }}\right)
=S_{\omega ^{\ast }}$ for $\tau ,\omega \in T\Lambda $. Then $\pi _{S}$ is
injective.
\end{enumerate}
\end{lemma}

\begin{proof}
Now we show (a). First we show that $\left\{ S_{\tau },S_{\omega ^{\ast
}}:\tau ,\omega \in T\Lambda \right\} $ satisfies (CP1). Take $x\in T\Lambda
^{0}$. We have to show $S_{x}=S_{x^{\ast }}=S_{x}^{2}$. Note that $%
S_{x}=F_{T,v}$ if $x=\alpha \left( v\right) $; and $S_{x}=T_{v}-F_{T,v}$,
otherwise. In both cases, by Lemma \ref{F_t,v}(a), we have $S_{x}=S_{x^{\ast
}}=S_{x}^{2}$, as required.

Now take $x,y\in T\Lambda ^{0}$ with $x\neq y$. We have to show $%
S_{x}S_{y}=0 $. Since $S_{x}$ is either $F_{T,v}$ or $T_{v}-F_{T,v}$; and $%
S_{y}$ is also either $F_{T,w}$ or $T_{w}-F_{T,w}$, then Lemma \ref{F_t,v}%
(b) tells that $x\neq y$ implies $S_{x}S_{y}=0$. Therefore, $\left\{ S_{\tau
},S_{\omega ^{\ast }}:\tau ,\omega \in T\Lambda \right\} $ satisfies (CP1).

Next we show that $\left\{ S_{\tau },S_{\omega ^{\ast }}:\tau ,\omega \in
T\Lambda \right\} $ satisfies (CP2). Take $\tau ,\omega \in T\Lambda $ where
$s\left( \tau \right) =r\left( \omega \right) $. We have to show $S_{\tau
}S_{\omega }=S_{\tau \omega }$ and $S_{\omega ^{\ast }}S_{\tau ^{\ast
}}=S_{\left( \tau \omega \right) ^{\ast }}$. Note that each $\tau $ and $%
\omega $ is either in the form $\alpha \left( \lambda \right) $ or $\beta
\left( \mu \right) $. So we give a separate argument for each case.

First suppose $\tau =\beta \left( \lambda \right) $. Since $s\left( \tau
\right) T\Lambda =\beta \left( s\left( \lambda \right) \right) $ and $%
s\left( \tau \right) =r\left( \omega \right) $, then $\omega =s\left( \beta
\left( \lambda \right) \right) $. Hence,
\begin{align}
S_{\beta \left( \lambda \right) }S_{\beta \left( s\left( \lambda \right)
\right) }& =\left( T_{\lambda }\left( T_{s\left( \lambda \right)
}-F_{T,s\left( \lambda \right) }\right) \right) \left( T_{s\left( \lambda
\right) }-F_{T,s\left( \lambda \right) }\right)  \label{equ-R2-Sbeta-Sbeta}
\\
& =T_{\lambda }\left( T_{s\left( \lambda \right) }-F_{T,s\left( \lambda
\right) }\right) ^{2}=T_{\lambda }\left( T_{s\left( \lambda \right)
}-F_{T,s\left( \lambda \right) }\right) =S_{\beta \left( \lambda \right) }%
\text{.}  \notag
\end{align}

Next suppose $\tau =\alpha \left( \lambda \right) $ and $\omega =\beta
\left( \mu \right) $. Then $s\left( \tau \right) =r\left( \omega \right) $
implies $\mu \in s\left( \lambda \right) \Lambda \backslash \left\{ s\left(
\lambda \right) \right\} $ and by Lemma \ref{F_t,v}(c), $F_{T,s\left(
\lambda \right) }T_{\mu }=T_{\mu }$. Hence,%
\begin{align}
S_{\alpha \left( \lambda \right) }S_{\beta \left( \mu \right) }& =\left(
T_{\lambda }F_{T,s\left( \lambda \right) }\right) \left( T_{\mu }\left(
T_{s\left( \mu \right) }-F_{T,s\left( \mu \right) }\right) \right)
=T_{\lambda }T_{\mu }\left( T_{s\left( \lambda \mu \right) }-F_{T,s\left(
\lambda \mu \right) }\right)  \label{equ-R2-Salpha-Sbeta} \\
& =T_{\lambda \mu }\left( T_{s\left( \lambda \mu \right) }-F_{T,s\left(
\lambda \mu \right) }\right) =S_{\beta \left( \lambda \mu \right) }\text{.}
\notag
\end{align}

Finally suppose $\tau =\alpha \left( \lambda \right) $ and $\omega =\alpha
\left( \mu \right) $. Then $S_{\beta \left( \lambda \right) }S_{\alpha
\left( \mu \right) }=\left( S_{\beta \left( \lambda \right) }S_{\beta \left(
s\left( \lambda \right) \right) }\right) \left( S_{\alpha \left( s\left(
\lambda \right) \right) }S_{\alpha \left( \mu \right) }\right) =0$ (since $%
S_{\beta \left( s\left( \lambda \right) \right) }S_{\alpha \left( s\left(
\lambda \right) \right) }=0$) and%
\begin{align}
S_{\alpha \left( \lambda \right) }S_{\alpha \left( \mu \right) }& =\left(
S_{\alpha \left( \lambda \right) }+S_{\beta \left( \lambda \right) }\right)
\left( S_{\alpha \left( \mu \right) }+S_{\beta \left( \mu \right) }\right)
-S_{\alpha \left( \lambda \right) }S_{\beta \left( \mu \right) }-S_{\beta
\left( \lambda \right) }S_{\alpha \left( \mu \right) }-S_{\beta \left(
\lambda \right) }S_{\beta \left( \mu \right) }  \label{equ-R2-Salpha-Salpha}
\\
& =T_{\lambda }T_{\mu }-S_{\alpha \left( \lambda \right) }S_{\beta \left(
\mu \right) }-S_{\beta \left( \lambda \right) }S_{\beta \left( \mu \right) }%
\text{.}  \notag
\end{align}%
If $\mu =s\left( \lambda \right) $, then $S_{\alpha \left( \lambda \right)
}S_{\beta \left( \mu \right) }=\left( S_{\alpha \left( \lambda \right)
}S_{\alpha \left( s\left( \lambda \right) \right) }\right) S_{\beta \left(
s\left( \lambda \right) \right) }=0$ (since $S_{\alpha \left( s\left(
\lambda \right) \right) }S_{\beta \left( s\left( \lambda \right) \right) }=0$%
) and by \eqref{equ-R2-Sbeta-Sbeta}, \eqref{equ-R2-Salpha-Salpha} becomes%
\begin{equation*}
S_{\alpha \left( \lambda \right) }S_{\alpha \left( s\left( \lambda \right)
\right) }=T_{\lambda }-S_{\beta \left( \lambda \right) }=S_{\alpha \left(
\lambda \right) }\text{.}
\end{equation*}%
On the other hand, if $\mu \neq s\left( \lambda \right) $, then $S_{\beta
\left( \lambda \right) }S_{\beta \left( \mu \right) }=\left( S_{\beta \left(
\lambda \right) }S_{\beta \left( s\left( \lambda \right) \right) }\right)
\left( S_{r\left( \beta \left( \mu \right) \right) }S_{\beta \left( \mu
\right) }\right) =0$ (since $\beta \left( s\left( \lambda \right) \right)
\neq r\left( \beta \left( \mu \right) \right) $and $S_{\beta \left( s\left(
\lambda \right) \right) }S_{r\left( \beta \left( \mu \right) \right) }=0$)
and by \eqref{equ-R2-Salpha-Sbeta}, \eqref{equ-R2-Salpha-Salpha} becomes%
\begin{equation*}
S_{\alpha \left( \lambda \right) }S_{\alpha \left( \mu \right) }=T_{\lambda
\mu }-S_{\beta \left( \lambda \mu \right) }=S_{\alpha \left( \lambda \mu
\right) }\text{.}
\end{equation*}%
Therefore, $S_{\tau }S_{\omega }=S_{\tau \omega }$ and by using a similar
argument, we get $S_{\omega ^{\ast }}S_{\tau ^{\ast }}=S_{\left( \tau \omega
\right) ^{\ast }}$. Thus $\left\{ S_{\tau },S_{\omega ^{\ast }}:\tau ,\omega
\in T\Lambda \right\} $ satisfies (CP2).

Now we show that $\left\{ S_{\tau },S_{\omega ^{\ast }}:\tau ,\omega \in
T\Lambda \right\} $ satisfies (CP3). Take $\tau ,\omega \in T\Lambda $. We
have to show $S_{\tau ^{\ast }}S_{\omega }=\sum_{(\rho ,\zeta )\in T\Lambda
^{\min }\left( \tau ,\omega \right) }S_{\rho }S_{\zeta ^{\ast }}$. Note that
each $\tau $ and $\omega $ is either in the form $\alpha \left( \lambda
\right) $ or $\beta \left( \mu \right) $. So we give a separate argument for
each case.

First suppose $\tau =\beta \left( \lambda \right) $. Since $s\left( \tau
\right) T\Lambda =\beta \left( s\left( \lambda \right) \right) $, then $%
T\Lambda ^{\min }\left( \tau ,\omega \right) \neq \emptyset $ implies $\operatorname{%
MCE}\left( \tau ,\omega \right) =\left\{ \tau \right\} $. Hence, if $%
T\Lambda ^{\min }\left( \tau ,\omega \right) \neq \emptyset $, then we have $%
\tau =\omega \beta \left( \nu \right) $ for some $\nu \in \Lambda $ and%
\begin{equation*}
S_{\tau ^{\ast }}S_{\omega }=S_{\beta \left( \nu \right) ^{\ast }}S_{\omega
^{\ast }}S_{\omega }=S_{\beta \left( \nu \right) ^{\ast }}S_{s\left( \omega
\right) }=S_{\beta \left( \nu \right) ^{\ast }}=\sum_{(\rho ,\zeta )\in
T\Lambda ^{\min }\left( \tau ,\omega \right) }S_{\rho }S_{\zeta ^{\ast }}.
\end{equation*}%
So suppose $T\Lambda ^{\min }\left( \tau ,\omega \right) =\emptyset $ and we
have to show $S_{\tau ^{\ast }}S_{\omega }=0$. First note that regardless of
whether $\omega $ is equal to $\alpha \left( \mu \right) $ or $\beta \left(
\mu \right) $, $S_{\tau ^{\ast }}S_{\omega }$ has the form $\left( S_{\beta
\left( \lambda \right) ^{\ast }}T_{\mu }\right) b$. So it suffices to show $%
S_{\beta \left( \lambda \right) ^{\ast }}T_{\mu }=0$. We have
\begin{equation}
S_{\beta \left( \lambda \right) ^{\ast }}T_{\mu }=\left( T_{s\left( \lambda
\right) }-F_{T,s\left( \lambda \right) }\right) T_{\lambda ^{\ast }}T_{\mu
}=\left( T_{s\left( \lambda \right) }-F_{T,s\left( \lambda \right) }\right)
\sum_{(\nu ,\gamma )\in \Lambda ^{\min }\left( \lambda ,\mu \right) }T_{\nu
}T_{\gamma ^{\ast }}\text{.}  \label{equ-R3-S_beta(lambda)*T_mu}
\end{equation}%
If $\Lambda ^{\min }\left( \lambda ,\mu \right) =\emptyset $, then $S_{\beta
\left( \lambda \right) ^{\ast }}T_{\mu }=0$, as required. So suppose $%
\Lambda ^{\min }\left( \lambda ,\mu \right) \neq \emptyset $. Since $%
T\Lambda ^{\min }\left( \tau ,\omega \right) =\emptyset $ and $\Lambda
^{\min }\left( \lambda ,\mu \right) \neq \emptyset $, then $\lambda \notin
\operatorname{MCE}\left( \lambda ,\mu \right) $. Hence, for every $(\nu ,\gamma )\in
\Lambda ^{\min }\left( \lambda ,\mu \right) $, we have $\nu \in s\left(
\lambda \right) \Lambda \backslash \left\{ s\left( \lambda \right) \right\} $
and by Lemma \ref{F_t,v}(c), $F_{T,s\left( \lambda \right) }T_{\nu }=T_{\nu
} $. Hence, we can rewrite \eqref{equ-R3-S_beta(lambda)*T_mu} as%
\begin{equation*}
S_{\beta \left( \lambda \right) ^{\ast }}T_{\mu }=\sum_{(\nu ,\gamma )\in
\Lambda ^{\min }\left( \lambda ,\mu \right) }\left( T_{\nu }-T_{\nu }\right)
T_{\gamma ^{\ast }}=0\text{,}
\end{equation*}%
as required.

Next suppose $\tau =\alpha \left( \lambda \right) $ and $\omega =\beta
\left( \mu \right) $. By using a similar argument as in the case $\tau
=\beta \left( \lambda \right) $, we get $S_{\alpha \left( \lambda \right)
^{\ast }}S_{\beta \left( \mu \right) }=\sum_{(\rho ,\zeta )\in T\Lambda
^{\min }\left( \alpha \left( \lambda \right) ,\beta \left( \mu \right)
\right) }S_{\rho }S_{\zeta ^{\ast }}$.

Finally suppose $\tau =\alpha \left( \lambda \right) $ and $\omega =\alpha
\left( \mu \right) $. We give a separate argument for whether $\alpha \left(
\lambda \right) $ or $\alpha \left( \mu \right) $ belongs to $\operatorname{MCE}%
\left( \alpha \left( \lambda \right) ,\alpha \left( \mu \right) \right) $.
First suppose that at least one of $\alpha \left( \lambda \right) $ and $%
\alpha \left( \mu \right) $ belongs to $\operatorname{MCE}\left( \alpha \left(
\lambda \right) ,\alpha \left( \mu \right) \right) $. Without loss of
generality, we suppose $\alpha \left( \lambda \right) \in \operatorname{MCE}\left(
\alpha \left( \lambda \right) ,\alpha \left( \mu \right) \right) $. [A
similar argument also applies when $\alpha \left( \mu
\right) \in \operatorname{MCE}\left( \alpha \left( \lambda \right) ,\alpha \left(
\mu \right) \right) $.] Then $\alpha \left( \lambda \right) =\alpha \left(
\mu \nu \right) $ for some $\nu \in \Lambda $ and
\begin{equation*}
S_{\alpha \left( \lambda \right) ^{\ast }}S_{\alpha \left( \mu \right)
}=S_{\alpha \left( \nu \right) ^{\ast }}S_{\alpha \left( \mu \right) ^{\ast
}}S_{\alpha \left( \mu \right) }=S_{\alpha \left( \nu \right) ^{\ast
}}=\sum_{(\rho ,\zeta )\in T\Lambda ^{\min }\left( \alpha \left( \lambda
\right) ,\alpha \left( \mu \right) \right) }S_{\rho }S_{\zeta ^{\ast }}\text{%
,}
\end{equation*}%
as required.

So suppose $\alpha \left( \lambda \right) ,\alpha \left( \mu \right) \notin
\operatorname{MCE}\left( \alpha \left( \lambda \right) ,\alpha \left( \mu \right)
\right) $. Hence $\lambda ,\mu \notin \operatorname{MCE}\left( \lambda ,\mu \right) $%
. Then for every $(\nu ,\gamma )\in \Lambda ^{\min }\left( \lambda ,\mu
\right) $, we have $\nu \in s\left( \lambda \right) \Lambda \backslash
\left\{ s\left( \lambda \right) \right\} $ and $\gamma \in s\left( \mu
\right) \Lambda \backslash \left\{ s\left( \mu \right) \right\} $, and by
Lemma \ref{F_t,v}(c), $F_{T,s\left( \lambda \right) }T_{\nu }=T_{\nu }$ and $%
T_{\gamma ^{\ast }}F_{T,s\left( \mu \right) }=T_{\gamma ^{\ast }}$.
Therefore,%
\begin{align}
S_{\alpha \left( \lambda \right) ^{\ast }}S_{\alpha \left( \mu \right) }&
=\left( F_{T,s\left( \lambda \right) }T_{\lambda ^{\ast }}\right) \left(
T_{\mu }F_{T,s\left( \mu \right) }\right) =F_{T,s\left( \lambda \right) }%
\Big(\sum_{(\nu ,\gamma )\in \Lambda ^{\min }\left( \lambda ,\mu \right)
}T_{\nu }T_{\gamma ^{\ast }}\Big)F_{T,s\left( \mu \right) }
\label{equ-R3-Salpha-Salpha} \\
& =\sum_{(\nu ,\gamma )\in \Lambda ^{\min }\left( \lambda ,\mu \right)
}\left( F_{T,s\left( \lambda \right) }T_{\nu }\right) \left( T_{\gamma
^{\ast }}F_{T,s\left( \mu \right) }\right) =\sum_{(\nu ,\gamma )\in \Lambda
^{\min }\left( \lambda ,\mu \right) }T_{\nu }T_{\gamma ^{\ast }}\text{.}
\notag
\end{align}%
Since $s\left( \nu \right) =s\left( \gamma \right) $ for every $(\nu ,\gamma
)\in \Lambda ^{\min }\left( \lambda ,\mu \right) $, then by %
Lemma \ref{F_t,v}(a), \eqref{equ-R3-Salpha-Salpha} becomes%
\begin{align*}
S_{\alpha \left( \lambda \right) ^{\ast }}& S_{\alpha \left( \mu \right) } \\
& =\sum_{(\nu ,\gamma )\in \Lambda ^{\min }\left( \lambda ,\mu \right)
}\left( \left( T_{\nu }F_{T,s\left( \nu \right) }\right) \left( F_{T,s\left(
\gamma \right) }T_{\gamma ^{\ast }}\right) +T_{\nu }\left( T_{s\left( \nu
\right) }-F_{T,s\left( \nu \right) }\right) \left( T_{s\left( \gamma \right)
}-F_{T,s\left( \gamma \right) }\right) T_{\gamma ^{\ast }}\right) \\
& =\sum_{(\nu ,\gamma )\in \Lambda ^{\min }\left( \lambda ,\mu \right)
}\left( S_{\alpha \left( \nu \right) }S_{\alpha \left( \gamma \right) ^{\ast
}}+S_{\beta \left( \nu \right) }S_{\beta \left( \gamma \right) ^{\ast
}}\right) =\sum\limits_{(\rho ,\zeta )\in T\Lambda ^{\min }\left( \alpha
\left( \lambda \right) ,\alpha \left( \mu \right) \right) }S_{\rho }S_{\zeta
^{\ast }}\text{,}
\end{align*}%
as required. Therefore, $S_{\tau ^{\ast }}S_{\omega }=\sum_{(\rho ,\zeta
)\in T\Lambda ^{\min }\left( \tau ,\omega \right) }S_{\rho }S_{\zeta ^{\ast
}}$ for all $\tau ,\omega \in T\Lambda $. Thus the collection $\{S_{\tau
},S_{\omega ^{\ast }}:\tau ,\omega \in T\Lambda \}$ satisfies (CP3).

\medskip

To show that $\left\{ S_{\tau },S_{\omega ^{\ast }}:\tau ,\omega \in
T\Lambda \right\} $ is a Kumjian-Pask $T\Lambda $-family, by Lemma \ref%
{KP-TLambda-when-lambda-has-no-sources}, it suffices to show $\prod_{g\in
\alpha \left( v\right) T\Lambda ^{1}}$ $\left( S_{\alpha \left( v\right)
}-S_{g}S_{g^{\ast }}\right) =0$ for $\alpha \left( v\right) \in T\Lambda
^{0} $. Take $\alpha \left( v\right) \in T\Lambda ^{0}$ and we have
\begin{align}
\prod\limits_{g\in \alpha \left( v\right) T\Lambda ^{1}}& \left( S_{\alpha
\left( v\right) }-S_{g}S_{g^{\ast }}\right)  \label{equ-KP-multiplication} \\
& =\prod\limits_{e\in v\Lambda ^{1}}\left( S_{\alpha \left( v\right)
}-S_{\alpha \left( e\right) }S_{\alpha \left( e\right) ^{\ast }}\right)
\left( S_{\alpha \left( v\right) }-S_{\beta \left( e\right) }S_{\beta \left(
e\right) ^{\ast }}\right)  \notag \\
& =\prod\limits_{e\in v\Lambda ^{1}}\Big(T_{v}F_{T,v}-T_{e}F_{T,s\left(
e\right) }^{2}T_{e^{\ast }}\Big)\Big(T_{v}F_{T,v}-T_{e}\left( T_{s\left(
e\right) }-F_{T,s\left( e\right) }\right) ^{2}T_{e^{\ast }}\Big)  \notag \\
& =\prod\limits_{e\in v\Lambda ^{1}}\left( F_{T,v}-T_{e}\left( T_{s\left(
e\right) }-F_{T,s\left( e\right) }\right) T_{e^{\ast }}-T_{e}F_{T,s\left(
e\right) }T_{e^{\ast }}+T_{e}F_{T,s\left( e\right) }\left( T_{s\left(
e\right) }-F_{T,s\left( e\right) }\right) T_{e^{\ast }}\right)  \notag
\end{align}%
by Lemma \ref{F_t,v}(a,c). Since $F_{T,s\left( e\right) }\left( T_{s\left(
e\right) }-F_{T,s\left( e\right) }\right) =0$ (by Lemma \ref{F_t,v}(c)),
then \eqref{equ-KP-multiplication} becomes%
\begin{align*}
\prod\limits_{g\in \alpha \left( v\right) T\Lambda ^{1}}\left( S_{\alpha
\left( v\right) }-S_{g}S_{g^{\ast }}\right) & =\prod\limits_{e\in v\Lambda
^{1}}\left( F_{T,v}-T_{e}T_{e^{\ast }}\right) \\
& =F_{T,v}\prod\limits_{e\in v\Lambda ^{1}}\left( T_{v}-T_{e}T_{e^{\ast
}}\right) \text{ (by Lemma \ref{F_t,v}(a,c))} \\
& =F_{T,v}\left( T_{v}-F_{T,v}\right) =F_{T,v}T_{v}-F_{T,v}^{2} \\
& =F_{T,v}-F_{T,v}\text{ (by Lemma \ref{F_t,v}(a,c))} \\
& =0\text{.}
\end{align*}%
Then $\left\{ S_{\tau },S_{\omega ^{\ast }}:\tau ,\omega \in T\Lambda
\right\} $ is a Kumjian-Pask $T\Lambda $-family, as required.

Next we show (b). Suppose that $rT_{v}\neq 0$ and $r\prod_{e\in v\Lambda
^{1}}\left( T_{v}-T_{e}T_{e}^{\ast }\right) \neq 0$ for all $v\in \Lambda
^{0}$, and $\pi _{S}:{\normalsize \operatorname{KP}}_{R}\left( T\Lambda \right)
\rightarrow {\normalsize A}$ is the $R$-algebra homomorphism such that $\pi
_{S}\left( s_{\tau }\right) =S_{\tau }$ and $\pi _{S}\left( s_{\omega ^{\ast
}}\right) =S_{\omega ^{\ast }}$ for $\tau ,\omega \in T\Lambda $. We have to
show $\pi _{S}$ is injective. Since $rT_{v}\neq 0$ for all $r\in \left.
R\right\backslash \left\{ 0\right\} $ and $v\in \Lambda ^{0}$, then by Lemma %
\ref{F_t,v}(d), $rF_{T,v}\neq 0$ for all $r\in \left. R\right\backslash
\left\{ 0\right\} $ and $v\in \Lambda ^{0}$. Therefore, for all $r\in \left.
R\right\backslash \left\{ 0\right\} $ and $v\in \Lambda ^{0}$,%
\begin{equation*}
rS_{\alpha \left( v\right) }=rT_{v}F_{T,v}=rF_{T,v}\neq 0
\end{equation*}%
and%
\begin{equation*}
rS_{\beta \left( v\right) }=rT_{v}\left( T_{v}-F_{T,v}\right) =r\left(
T_{v}-F_{T,v}\right) =r\prod_{e\in v\Lambda ^{1}}\left(
T_{v}-T_{e}T_{e^{\ast }}\right) \neq 0.
\end{equation*}%
Hence, $rS_{x}\neq 0$ for all $r\in \left. R\right\backslash \left\{
0\right\} $ and $x\in T\Lambda ^{0}$. Since $T\Lambda $ is aperiodic, then
by Theorem \ref{the-CK-uniqueness-theorem}, $\pi _{S}$ is injective.
\end{proof}

One immediate application of Lemma \ref{S-is-a-KP-family} is:

\begin{theorem}
\label{CP-is-isomorphic-to-KP}Let $\Lambda $ be a row-finite $k$-graph with
no sources and $R$ be a commutative ring with $1$. Suppose that $%
\{t_{\lambda },t_{\mu ^{\ast }}:\lambda ,u\in \Lambda \}$ is the universal
Cohn $\Lambda $-family and $\left\{ s_{\tau },s_{\omega ^{\ast }}:\tau
,\omega \in T\Lambda \right\} $ is the universal Kumjian-Pask $\,T\Lambda $%
-family. For $\tau ,\omega \in T\Lambda $, define%
\begin{align*}
S_{\tau }& :=\left\{
\begin{array}{ccc}
t_{\lambda }F_{t,s\left( \lambda \right) } & \text{if} & \tau =\alpha \left(
\lambda \right) \text{;} \\
t_{\lambda }\left( t_{s\left( \lambda \right) }-F_{t,s\left( \lambda \right)
}\right) & \text{if} & \tau =\beta \left( \lambda \right) \text{,}%
\end{array}%
\right. \\
S_{\omega ^{\ast }}& :=\left\{
\begin{array}{ccc}
F_{t,s\left( \mu \right) }t_{\mu ^{\ast }} & \text{if} & \omega =\alpha
\left( \mu \right) \text{;} \\
\left( t_{s\left( \mu \right) }-F_{t,s\left( \mu \right) }\right) t_{\mu
^{\ast }} & \text{if} & \omega =\beta \left( \mu \right) \text{.}%
\end{array}%
\right. \text{ }
\end{align*}%
Then

\begin{enumerate}
\item[(a)] There exists an $R$-algebra homomorphism $\pi :{\normalsize \operatorname{%
KP}}_{R}\left( T\Lambda \right) \rightarrow {\normalsize \operatorname{C}}_{R}\left(
\Lambda \right) $ such that $\pi \left( s_{\tau }\right) =S_{\tau }$ and $%
\pi \left( s_{\omega ^{\ast }}\right) =S_{\omega ^{\ast }}$ for $\tau
,\omega \in T\Lambda $. Furthermore, $\pi $ is an isomorphism.

\item[(b)] The subsets%
\begin{equation*}
{\normalsize \operatorname{C}}_{R}\left( \Lambda \right) _{n}:=\operatorname{span}%
_{R}\left\{ t_{\lambda }t_{\mu ^{\ast }}:\lambda ,\mu \in \Lambda ,d\left(
\lambda \right) -d\left( \mu \right) =n\right\}
\end{equation*}%
form a $\mathbb{Z}^{k}$-grading of ${\normalsize \operatorname{C}}_{R}\left( \Lambda
\right) $.
\end{enumerate}
\end{theorem}

\begin{proof}
First we show part (a). By Lemma \ref{S-is-a-KP-family}(a), $\left\{ S_{\tau
},S_{\omega ^{\ast }}:\tau ,\omega \in T\Lambda \right\} $ is a Kumjian-Pask
$T\Lambda $-family and by the universal property of Kumjian-Pask $T\Lambda $%
-family \cite[Theorem 3.7(a)]{CP15}, there exists an $R$-algebra
homomorphism $\pi :{\normalsize \operatorname{KP}}_{R}\left( T\Lambda \right)
\rightarrow {\normalsize \operatorname{C}}_{R}\left( \Lambda \right) $ such that $%
\pi \left( s_{\tau }\right) =S_{\tau }$ and $\pi \left( s_{\omega ^{\ast
}}\right) =S_{\omega ^{\ast }}$ for $\tau ,\omega \in T\Lambda $. On the
other hand, Theorem \ref{universal-CP-family}(b) tells $rt_{v}\neq 0$ and $%
r\prod_{e\in v\Lambda ^{1}}\left( t_{v}-t_{e}t_{e}^{\ast }\right) \neq 0$
for all $r\in \left. R\right\backslash \left\{ 0\right\} $ and $v\in \Lambda
^{0}$. Hence, by Lemma \ref{S-is-a-KP-family}(b), $\pi $ is injective.

Now we show the surjectivity of $\pi $. Since
\begin{equation*}
{\normalsize \operatorname{C}}_{R}\left( \Lambda \right) =\operatorname{span}%
_{R}\{t_{\lambda }t_{\mu ^{\ast }}:\lambda ,u\in \Lambda ,s\left( \lambda
\right) =s\left( \mu \right) \}
\end{equation*}%
(Proposition \ref{properties-of-CP}(b)), it suffices to show that for $%
\lambda ,\mu \in \Lambda $, both $t_{\lambda }$ and $t_{\mu ^{\ast }}$
belong to the image of $\pi $. Take $\lambda ,\mu \in \Lambda $, then we
have
\begin{align}
t_{\lambda }& =t_{\lambda }t_{s\left( \lambda \right) }=t_{\lambda
}F_{t,s\left( \lambda \right) }+t_{\lambda }t_{s\left( \lambda \right)
}-t_{\lambda }F_{t,s\left( \lambda \right) }  \label{eq-surjectivity-t} \\
& =t_{\lambda }F_{t,s\left( \lambda \right) }+t_{\lambda }\left( t_{s\left(
\lambda \right) }-F_{t,s\left( \lambda \right) }\right) =\pi \left(
s_{\alpha \left( \lambda \right) }\right) +\pi \left( s_{\beta \left(
\lambda \right) }\right)  \notag
\end{align}%
and%
\begin{align}
t_{\mu ^{\ast }}& =t_{s\left( \mu \right) }t_{\mu ^{\ast }}=F_{t,s\left( \mu
\right) }t_{\mu ^{\ast }}+t_{s\left( \mu \right) }t_{\mu ^{\ast
}}-F_{t,s\left( \mu \right) }t_{\mu ^{\ast }}  \label{eq-surjectivity-t*} \\
& =F_{t,s\left( \mu \right) }t_{\mu ^{\ast }}+\left( t_{s\left( \mu \right)
}-F_{t,s\left( \mu \right) }\right) t_{\mu ^{\ast }}=\pi \left( s_{\alpha
\left( \mu \right) ^{\ast }}\right) +\pi \left( s_{\beta \left( \mu \right)
^{\ast }}\right) \text{,}  \notag
\end{align}%
as required. Therefore, $\pi $ is an isomorphism.

Next we show part (b). Recall from
\cite[Theorem~3.6(c)]{CP15} that the subsets%
\begin{equation*}
{\normalsize \operatorname{KP}}_{R}\left( T\Lambda \right) _{n}:=\operatorname{span}%
_{R}\left\{ s_{\tau }s_{\omega ^{\ast }}:\tau ,\omega \in T\Lambda ,d\left(
\tau \right) -d\left( \omega \right) =n\right\}
\end{equation*}%
forms a $\mathbb{Z}^{k}$-grading of ${\normalsize \operatorname{KP}}_{R}\left(
\Lambda \right) $. Note that for every $v\in \Lambda ^{0}$, $d\left(
t_{s\left( \lambda \right) }-F_{t,s\left( \lambda \right) }\right)
=0=d\left( F_{t,s\left( \lambda \right) }\right) $. Hence regardless of
whether $\tau $ and $\omega $ are in the form $\alpha \left( \lambda \right)
$ or $\beta \left( \mu \right) $, we have $d\left( \tau \right) -d\left(
\omega \right) =d\left( \lambda \right) -d\left( \mu \right) $ and $s_{\tau
}s_{\omega ^{\ast }}\in {\normalsize \operatorname{C}}_{R}\left( \Lambda \right)
_{n} $ which implies $\pi \left( s_{\tau }s_{\omega ^{\ast }}\right) \in
{\normalsize \operatorname{KP}}_{R}\left( T\Lambda \right) _{n}$. Since $\pi $ is an
isomorphism, then ${\normalsize \operatorname{C}}_{R}\left( \Lambda \right) _{n}$
forms a grading for ${\normalsize \operatorname{C}}_{R}\left( \Lambda \right) $, as
required.
\end{proof}

\begin{remark}
Our Theorem~\ref{CP-is-isomorphic-to-KP} generalises results about Cohn path
algebras associated to 1-graphs. In particular, Theorem \ref%
{CP-is-isomorphic-to-KP}(a) generalises \cite[Theorem 5]{AK13} (which is
also stated in \cite[Theorem 1.5.18]{Leavitt path algebras}); Theorem \ref%
{CP-is-isomorphic-to-KP}(b) generalises \cite[Corollary 2.1.5.(ii)]{Leavitt
path algebras}.
\end{remark}

\begin{proof}[Proof of Theorem \protect\ref%
{the-uniqueness-theorem-of-CP-family}]
Since $\left\{ t_{\lambda },t_{\mu ^{\ast }}:\lambda ,u\in \Lambda \right\} $
is a Cohn $\Lambda $-family, then so is $\{\phi \left( t_{\lambda }\right)
,\phi \left( t_{\mu ^{\ast }}\right) :\lambda ,u\in \Lambda \}$. For $\tau
,\omega \in T\Lambda $, define%
\begin{align*}
S_{\tau }& :=\left\{
\begin{array}{ccc}
\phi \left( t_{\lambda }\right) F_{\phi \left( t\right) ,s\left( \lambda
\right) } & \text{if} & \tau =\alpha \left( \lambda \right) \text{;} \\
\phi \left( t_{\lambda }\right) \left( \phi \left( t_{s\left( \lambda
\right) }\right) -F_{\phi \left( t\right) ,s\left( \lambda \right) }\right)
& \text{if} & \tau =\beta \left( \lambda \right) \text{,}%
\end{array}%
\right. \\
S_{\omega ^{\ast }}& :=\left\{
\begin{array}{ccc}
F_{\phi \left( t\right) ,s\left( u\right) }\phi \left( t_{\mu ^{\ast
}}\right) & \text{if} & \omega =\alpha \left( \mu \right) \text{;} \\
\left( \phi \left( t_{s\left( u\right) }\right) -F_{\phi \left( t\right)
,s\left( u\right) }\right) \phi \left( t_{\mu ^{\ast }}\right) & \text{if} &
\omega =\beta \left( \mu \right) \text{.}%
\end{array}%
\right. \text{ }
\end{align*}%
Lemma \ref{S-is-a-KP-family}(a) tells that $\left\{ S_{\tau },S_{\omega
^{\ast }}:\tau ,\omega \in T\Lambda \right\} $ is a Kumjian-Pask $T\Lambda $%
-family and by the universal property of Kumjian-Pask $T\Lambda $-family,
there exists an $R$-algebra homomorphism $\pi _{S}:{\normalsize \operatorname{KP}}%
_{R}\left( \Lambda \right) \rightarrow {\normalsize A}$ such that $\pi
_{S}\left( s_{\tau }\right) =S_{\tau }$ and $\pi _{S}\left( s_{\omega ^{\ast
}}\right) =S_{\omega ^{\ast }}$ for $\tau ,\omega \in T\Lambda $. On the
other hand, by the hypothesis, $\phi \left( rt_{v}\right) \neq 0$ and $\phi
\left( r\prod_{e\in v\Lambda ^{1}}\left( t_{v}-t_{e}t_{e}^{\ast }\right)
\right) \neq 0$ for all $r\in \left. R\right\backslash \left\{ 0\right\} $
and $v\in \Lambda ^{0}$. Hence, by Lemma \ref{S-is-a-KP-family}(b), $\pi
_{S} $\ is injective.

Now recall from Theorem \ref{CP-is-isomorphic-to-KP}(a) that $\pi :%
{\normalsize \operatorname{KP}}_{R}\left( T\Lambda \right) \rightarrow {\normalsize
\operatorname{C}}_{R}\left( \Lambda \right) $ is an isomorphism with
\begin{align*}
\pi \left( s_{\tau }\right) & =\left\{
\begin{array}{ccc}
t_{\lambda }F_{t,s\left( \lambda \right) } & \text{if} & \tau =\alpha \left(
\lambda \right) \text{;} \\
t_{\lambda }\left( t_{s\left( \lambda \right) }-F_{t,s\left( \lambda \right)
}\right) & \text{if} & \tau =\beta \left( \lambda \right) \text{,}%
\end{array}%
\right. \\
\pi \left( s_{\omega ^{\ast }}\right) & =\left\{
\begin{array}{ccc}
F_{t,s\left( \mu \right) }t_{\mu ^{\ast }} & \text{if} & \omega =\alpha
\left( \mu \right) \text{;} \\
\left( t_{s\left( \mu \right) }-F_{t,s\left( \mu \right) }\right) t_{\mu
^{\ast }} & \text{if} & \omega =\beta \left( \mu \right) \text{.}%
\end{array}%
\right. \text{ }
\end{align*}%
Note that for $\lambda ,\mu \in \Lambda $, we have $t_{\lambda }=\pi \left(
s_{\alpha \left( \lambda \right) }\right) +\pi \left( s_{\beta \left(
\lambda \right) }\right) $ and $t_{\mu ^{\ast }}=\pi \left( s_{\alpha \left(
\mu \right) ^{\ast }}\right) +\pi \left( s_{\beta \left( \mu \right) ^{\ast
}}\right) $ (see \eqref{eq-surjectivity-t} and \eqref{eq-surjectivity-t*}).
Hence,%
\begin{align*}
\left( \pi _{S}\circ \pi ^{-1}\right) \left( t_{\lambda }\right) & =\left(
\pi _{S}\circ \pi ^{-1}\right) \left( \pi \left( s_{\alpha \left( \lambda
\right) }\right) +\pi \left( s_{\beta \left( \lambda \right) }\right)
\right) =\pi _{S}\left( s_{\alpha \left( \lambda \right) }\right) +\pi
_{S}\left( s_{\alpha \left( \lambda \right) }\right) \\
& =S_{\alpha \left( \lambda \right) }+S_{\alpha \left( \lambda \right)
}=\phi \left( t_{\lambda }\right) F_{\phi \left( t\right) ,s\left( \lambda
\right) }+\phi \left( t_{\lambda }\right) \left( \phi \left( t_{s\left(
\lambda \right) }\right) -F_{\phi \left( t\right) ,s\left( \lambda \right)
}\right) \\
& =\phi \left( t_{\lambda }\right) \text{.}
\end{align*}%
and%
\begin{align*}
\left( \pi _{S}\circ \pi ^{-1}\right) \left( t_{\mu ^{\ast }}\right) &
=\left( \pi _{S}\circ \pi ^{-1}\right) \left( \pi \left( s_{\alpha \left(
\mu \right) ^{\ast }}\right) +\pi \left( s_{\beta \left( \mu \right) ^{\ast
}}\right) \right) =\pi _{S}\left( s_{\alpha \left( \mu \right) ^{\ast
}}\right) +\pi _{S}\left( s_{\beta \left( \mu \right) ^{\ast }}\right) \\
& =S_{\alpha \left( \mu \right) ^{\ast }}+S_{\alpha \left( \mu \right)
^{\ast }}=F_{\phi \left( t\right) ,s\left( \mu \right) }\phi \left( t_{\mu
^{\ast }}\right) +\left( \phi \left( t_{s\left( \mu \right) }\right)
-F_{\phi \left( t\right) ,s\left( \mu \right) }\right) \phi \left( t_{\mu
^{\ast }}\right) \\
& =\phi \left( t_{\mu ^{\ast }}\right) \text{.}
\end{align*}%
These imply $\phi =\pi _{S}\circ \pi ^{-1}$ since ${\normalsize \operatorname{C}}%
_{R}\left( \Lambda \right) =\operatorname{span}_{R}\{t_{\lambda }t_{\mu ^{\ast
}}:\lambda ,u\in \Lambda ,s\left( \lambda \right) =s\left( \mu \right) \}$
(Proposition \ref{properties-of-CP}(b)). Furthermore, the injectivity of
both $\pi ^{-1}$ and $\pi _{S}$ imply that $\phi $ is also injective.
\end{proof}

\begin{remark}
The Cohn $\Lambda $-family $\left\{ T_{\lambda },T_{\mu ^{\ast }}:\lambda
,u\in \Lambda \right\} $ as constructed in Proposition \ref%
{the-path-representation} satisfies $rT_{v}\neq 0$ and $r\prod_{e\in
v\Lambda ^{1}}\left( T_{v}-T_{e}T_{e}^{\ast }\right) \neq 0$ for all $r\in
\left. R\right\backslash \left\{ 0\right\} $ and $v\in \Lambda ^{0}$. Hence
Theorem \ref{the-uniqueness-theorem-of-CP-family} tells that the $R$-algebra
homomorphism $\phi _{T}:{\normalsize \operatorname{C}}_{R}\left( \Lambda \right)
\rightarrow \operatorname{End}\left( \mathbb{F}_{R}\left( W_{\Lambda }\right)
\right) $ such that $\phi _{T}\left( t_{\lambda }\right) =T_{\lambda }$ and $%
\phi _{T}\left( t_{\mu ^{\ast }}\right) =T_{\mu ^{\ast }}$ for $\lambda ,\mu
\in \Lambda $, is injective.
\end{remark}

\begin{remark}
Note that when $\Lambda$ is a 1-graph, that is, when $k=1$, then $\Lambda$
is the path category of a directed graph $E$. One consequence of Theorem~\ref%
{universal-CP-family} and Theorem~\ref{the-CK-uniqueness-theorem} is that
the universal Cohn algebra $\operatorname{C}_{R}(\Lambda)$ that we have constructed
is isomorphic to the Cohn path algebra associated to $E$ as defined in \cite[%
Definition~1.5.1]{Leavitt path algebras}. Since \cite[Definition~1.5.1]%
{Leavitt path algebras} only considers the situation where $R$ is a field,
our construction gives a generalisation of the Cohn path algebra to the
setting where $R$ is an arbitrary commutative ring with 1.
\end{remark}

\section{Examples and Applications}

\label{Section-examples-and-applications}

\subsection{Higher-rank graph Toeplitz algebras.}

As mentioned in the introduction of Section \ref%
{Section-the-uniqueness-theorem-of-CP-family}, the uniqueness theorem for
Cohn path algebras (Theorem \ref{the-uniqueness-theorem-of-CP-family}) is an
analogue of the uniqueness theorem for Toeplitz algebras \cite[Theorem 8.1]%
{RS05}. We show that if $\Lambda $ is a row-finite $k$-graph with no
sources, then its Cohn path algebra over the complex numbers is dense in the
Toeplitz algebra associated to $\Lambda $ (Proposition \ref%
{CP-is-dense-in-TC}). First we give some preliminaries on
Toeplitz-Cuntz-Krieger $\Lambda $-families and Toeplitz algebras.

Suppose that $\Lambda $ is a row-finite $k$-graph with no sources. A \emph{%
Toeplitz-Cuntz-Krieger }$\Lambda $\emph{-family} is a collection $\left\{
Q_{\lambda }:\lambda \in \Lambda \right\} $ of partial isometries in a $%
C^{\ast }$-algebra $B$ satisfying:

\begin{enumerate}
\item[(TCK1)] $\left\{ Q_{v}:v\in \Lambda ^{0}\right\} $ is a collection of
mutually orthogonal projections;

\item[(TCK2)] $Q_{\lambda }Q_{\mu }=Q_{\lambda \mu }$ whenever $s\left(
\lambda \right) =r\left( \mu \right) $; and

\item[(TCK3)] $Q_{\lambda }^{\ast }Q_{\mu }=\sum_{(\nu ,\gamma )\in \Lambda
^{\min }\left( \lambda ,\mu \right) }Q_{\nu }Q_{\gamma }^{\ast }$ for all $%
\lambda ,\mu \in \Lambda $.
\end{enumerate}

\begin{remark}
In \cite[Lemma 9.2]{RS05}, a Toeplitz-Cuntz-Krieger $\Lambda $-family is
defined to be a collection of partial isometries $\left\{ Q_{\lambda
}:\lambda \in \Lambda \right\} $ satisfying (TCK1-3) and an additional
condition: for all $m\in \mathbb{N}^{k}\backslash \left\{ 0\right\} $, $v\in
\Lambda ^{0}$ and every set $E\subseteq v\Lambda ^{m}$, $Q_{v}\geq
\sum_{\lambda \in E}Q_{\lambda }Q_{\lambda }^{\ast }$. However, by \cite[%
Lemma 2.7 (iii)]{RSY04}, this additional condition is a direct consequence
of (TCK1-3) and hence our definition is equivalent to that of \cite{RS05}.
\end{remark}

For a row-finite $k$-graph $\Lambda $, there exists a $C^{\ast }$-algebra $%
{\normalsize TC}^{\ast }\left( \Lambda \right) $, called the \emph{Toeplitz
algebra }of $\Lambda $, generated by the universal Toeplitz-Cuntz-Krieger $%
\Lambda $-family $\left\{ q_{\lambda }:\lambda \in \Lambda \right\} $.
Furthermore, for $v\in \Lambda ^{0}$, we have $q_{v}\neq 0$ and $\prod_{e\in
v\Lambda ^{1}}\left( q_{v}-q_{e}q_{e}^{\ast }\right) \neq 0$ \cite[Corollary
3.7.7]{Si03}.

\begin{proposition}
\label{CP-is-dense-in-TC}Let $\Lambda $ be a row-finite $k$-graph with no
sources. Suppose that $\left\{ q_{\lambda }:\lambda \in \Lambda \right\} $
is the universal Toeplitz-Cuntz-Krieger $\Lambda $-family and $\left\{
t_{\lambda },t_{\mu ^{\ast }}:\lambda ,u\in \Lambda \right\} $ is the
universal (complex) Cohn $\Lambda $-family. Then there is an isomorphism
\begin{equation*}
\phi _{q}:{\normalsize \operatorname{C}}_{\mathbb{C}}\left( \Lambda \right)
\rightarrow \operatorname{span}\{q_{\lambda }q_{\mu }^{\ast }:\lambda ,\mu \in
\Lambda \}
\end{equation*}%
such that $\phi _{q}\left( t_{\lambda }\right) =q_{\lambda }$ and $\phi
_{q}\left( t_{\mu ^{\ast }}\right) =q_{\mu }^{\ast }$ for $\lambda ,u\in
\Lambda $. In particular, ${\normalsize \operatorname{C}}_{\mathbb{C}}\left( \Lambda
\right) $ is isomorphic to a dense subalgebra of $TC^{\ast }\left( \Lambda
\right) $.
\end{proposition}

\begin{proof}
Since $\left\{ q_{\lambda }:\lambda \in \Lambda \right\} $ satisfies
(TCK1-3), then by putting $q_{\lambda }:=q_{\lambda }$ and $q_{\mu ^{\ast
}}:=q_{\mu }^{\ast }$, the collection $\left\{ q_{\lambda },q_{\mu ^{\ast
}}:\lambda ,\mu \in \Lambda \right\} $ is a Cohn $\Lambda $-family in $%
TC^{\ast }\left( \Lambda \right) $. Thus the universal property of $%
{\normalsize \operatorname{C}}_{\mathbb{C}}\left( \Lambda \right) $ gives a
homomorphism $\phi _{q}$ from ${\normalsize \operatorname{C}}_{\mathbb{C}}\left(
\Lambda \right) $ onto the dense subalgebra%
\begin{equation*}
A:=\operatorname{span}_{\mathbb{C}}\{q_{\lambda }q_{\mu }^{\ast }:\lambda ,\mu \in
\Lambda \}
\end{equation*}%
of $TC^{\ast }\left( \Lambda \right) $.

On the other hand, for all $r\in \left. \mathbb{C}\right\backslash \left\{
0\right\} $ and $v\in \Lambda ^{0}$, we have $\frac{1}{r}\phi _{q}\left(
rt_{v}\right) =q_{v}\neq 0$ and
\begin{equation*}
\frac{1}{r}\phi _{q}\Big(r\prod_{e\in v\Lambda ^{1}}\left(
t_{v}-t_{e}t_{e^{\ast }}\right) \Big)=\prod_{e\in v\Lambda ^{1}}\left(
q_{v}-q_{e}q_{e}^{\ast }\right) \neq 0\text{.}
\end{equation*}%
So $\phi _{q}\left( rt_{v}\right) \neq 0$ and $\phi _{q}\Big(r\prod_{e\in
v\Lambda ^{1}}\left( t_{v}-t_{e}t_{e^{\ast }}\right) \Big)\neq 0$ for all $%
r\in \left. \mathbb{C}\right\backslash \left\{ 0\right\} $ and $v\in \Lambda
^{0}$. Then by Theorem \ref{the-uniqueness-theorem-of-CP-family}, $\phi _{q}$
is injective.
\end{proof}

\begin{remark}
For $k=1$, Proposition \ref{CP-is-dense-in-TC} tells that the Cohn path
algebra of $1$-graph $E$ is isomorphic to a dense subalgebra of $TC^{\ast
}\left( E\right) $.
\end{remark}

\subsection{Groupoids and Steinberg algebras.}

In \cite[Proposition 5.4]{CP15}, the authors show that each Kumjian-Pask
algebra is isomorphic to a Steinberg algebra. Thus Theorem \ref%
{CP-is-isomorphic-to-KP} implies that the Cohn path algebra of $\Lambda $ is
also isomorphic to a Steinberg algebra associated to $T\Lambda $. However,
this is somewhat obscure because one has to go through $T\Lambda $. We
improve this result by showing that there exists a groupoid associated to $%
\Lambda $ such that its Steinberg algebra is isomorphic to the Cohn path
algebra of $\Lambda $ (Proposition \ref%
{CP-is-isomorphic-to-Steinberg-algebras}). We start out with an introduction
to groupoids and Steinberg algebras in general.

A groupoid $\mathcal{G}$ is a small category with inverses. We write $%
\mathcal{G}^{\left( 0\right) }$ for the set $\left\{ aa^{-1}:a\in \mathcal{G}%
\right\} $, and we denote $r$ and $s$ the range and source maps $r,s:%
\mathcal{G}\rightarrow \mathcal{G}^{\left( 0\right) }$. We also write $%
\mathcal{G}^{\left( 2\right) }$ for the set of pairs $\left( a,b\right) \in
\mathcal{G}\times \mathcal{G}$ with $s\left( a\right) =r\left( b\right) $,
and for $A,B\subseteq \mathcal{G}$,%
\begin{equation*}
AB:=\left\{ ab:a\in A,b\in B,\left( a,b\right) \in \mathcal{G}^{\left(
2\right) }\right\} \text{.}
\end{equation*}

We say $\mathcal{G}$ is \emph{topological }if $\mathcal{G}$ is endowed with
a topology such that the range, source, and composition maps are continuous.
We call an open set $U\subseteq \mathcal{G}$ \emph{open bisection }if $%
s|_{U} $ and $r|_{U}$ are homeomorhisms. Finally, a groupoid $\mathcal{G}$
is \emph{ample} if it has a basis of compact open bisections.

Suppose that $\mathcal{G}$ is Hausdorff ample groupoid and $R$ is a
commutative ring with $1$. The Steinberg algebra of $\mathcal{G}$, denoted $%
A_{R}\left( \mathcal{G}\right) $, is the set of all functions from $\mathcal{%
G}$ to $R$ that are locally constant and have compact support (see \cite%
{CE-M15,CFST14,St10}). Addition and scalar multiplication of $A_{R}\left(
\mathcal{G}\right) $ are defined pointwise, and convolution is given by%
\begin{equation*}
\left( f\star g\right) \left( a\right) :=\sum_{r\left( a\right) =r\left(
b\right) }f\left( b\right) g\left( b^{-1}a\right) \text{.}
\end{equation*}%
Furthermore, for compact open bisections $U$ and $V$, we have%
\begin{equation*}
1_{U}\star 1_{V}=1_{UV}\text{.}
\end{equation*}

\begin{example}
\label{groupouid-TGlambda}Suppose that $\Lambda $ is a row-finite $k$-graph
with no sources. Following \cite[Definition 3.4]{Y07}, we construct an ample
groupoid as follows. Write%
\begin{equation*}
\Lambda \ast _{s}\Lambda :=\left\{ \left( \lambda ,\mu \right) \in \Lambda
\times \Lambda :s\left( \lambda \right) =s\left( \mu \right) \right\} .
\end{equation*}%
For $\left( \lambda ,\mu \right) \in \Lambda \ast _{s}\Lambda $ and finite
subset $G\subseteq s\left( \lambda \right) \Lambda $, we write%
\begin{equation*}
TZ_{\Lambda }\left( \lambda \right) :=\lambda W_{\Lambda }\text{,}
\end{equation*}%
\begin{equation*}
TZ_{\Lambda }\left( \left. \lambda \right\backslash G\right) :=\left.
TZ_{\Lambda }\left( \lambda \right) \right\backslash \Big(\bigcup_{\nu \in
G}TZ_{\Lambda }\left( \lambda \nu \right) \Big)\text{,}
\end{equation*}%
\begin{align*}
TZ_{\Lambda }\left( \lambda \ast _{s}\mu \right) & :=\{\left( x,d\left(
\lambda \right) -d\left( \mu \right) ,y\right) \in \mathcal{TG}_{\Lambda
}:x\in TZ_{\Lambda }\left( \lambda \right) ,y\in TZ_{\Lambda }\left( \mu
\right) \\
& \quad \quad \quad \quad \quad \text{ and }\sigma ^{d\left( \lambda \right)
}x=\sigma ^{d\left( \mu \right) }y\}
\end{align*}%
and%
\begin{equation*}
TZ_{\Lambda }\left( \left. \lambda \ast _{s}\mu \right\backslash G\right)
:=\left. TZ_{\Lambda }\left( \lambda \ast _{s}\mu \right) \right\backslash %
\Big(\bigcup_{\nu \in G}TZ_{\Lambda }\left( \lambda \nu \ast _{s}\mu \nu
\right) \Big).
\end{equation*}%
Then $\mathcal{TG}_{\Lambda }$ is a groupoid, called the \emph{path groupoid}%
, with object set
\begin{equation*}
\operatorname{Obj}\left( \mathcal{TG}_{\Lambda }\right) :=W_{\Lambda }\text{,}
\end{equation*}%
morphism set
\begin{align*}
\operatorname{Mor}\left( \mathcal{TG}_{\Lambda }\right) & :=\{\left( \lambda
z,d\left( \lambda \right) -d\left( \mu \right) ,\mu z\right) \in W_{\Lambda
}\times \mathbb{Z}^{k}\times W_{\Lambda }: \\
& \quad \quad \quad \quad \quad \quad \left( \lambda ,\mu \right) \in
\Lambda \ast _{s}\Lambda ,z\in s\left( \lambda \right) W_{\Lambda }\} \\
& =\{\left( x,m,y\right) \in W_{\Lambda }\times \mathbb{Z}^{k}\times
W_{\Lambda }:\text{there exists }p,q\in \mathbb{N}^{k}\text{ such that} \\
& \quad \quad \quad \quad \quad \quad p\leq d\left( x\right) ,q\leq d\left(
y\right) ,p-q=m\text{ and }\sigma ^{p}x=\sigma ^{q}y\}\text{\thinspace ,}
\end{align*}%
range and source maps $r\left( x,m,y\right) :=x$ and $s\left( x,m,y\right)
:=y$, composition%
\begin{equation*}
\left( \left( x_{1},m_{1},y_{1}\right) ,\left( y_{1},m_{2},y_{2}\right)
\right) \mapsto \left( x_{1},m_{1}+m_{2},y_{2}\right) \text{,}
\end{equation*}%
and inversion $\left( x,m,y\right) \mapsto \left( y,-m,x\right) $.
Furthermore, the sets $TZ_{\Lambda }\left( \left. \lambda \right\backslash
G\right) $ form a basis of compact open sets for $\mathcal{TG}_{\Lambda
}^{\left( 0\right) }$, and the sets $TZ_{\Lambda }\left( \left. \lambda \ast
_{s}\mu \right\backslash G\right) $ form a basis of compact open sets for
locally-compact, second-countable, Hausdorff topology on $\mathcal{TG}%
_{\Lambda }$ under which it is an \'{e}tale topological groupoid. Since for $%
\left( \lambda ,\mu \right) \in \Lambda \ast _{s}\Lambda $ and finite subset
$G\subseteq s\left( \lambda \right) \Lambda $, $TZ_{\Lambda }\left( \left.
\lambda \ast _{s}\mu \right\backslash G\right) $ is a bijection, then $%
\mathcal{TG}_{\Lambda }$ is also ample.
\end{example}

\begin{remark}
We think of $\mathcal{TG}_{\Lambda }^{0}=W_{\Lambda }$ as a subset of $%
\mathcal{TG}_{\Lambda }$ under the correspondence $x\mapsto \left(
x,0,x\right) $.
\end{remark}

\begin{proposition}
Let $\Lambda $ be a row-finite $k$-graph with no sources. Then the path
groupoid $\mathcal{TG}_{\Lambda }$ is \emph{effective}, in the sense that
the interior of%
\begin{equation*}
\operatorname{Iso}\left( \mathcal{TG}_{\Lambda }\right) :=\left\{ a\in \mathcal{TG}%
_{\Lambda }:s\left( a\right) =r\left( a\right) \right\}
\end{equation*}%
is $\mathcal{TG}_{\Lambda }^{\left( 0\right) }$.
\end{proposition}

\begin{proof}
For $x\in \mathcal{TG}_{\Lambda }^{\left( 0\right) }$, we have $\left(
x,0,x\right) \in \operatorname{Iso}\left( \mathcal{TG}_{\Lambda }\right) $ and then $%
\mathcal{TG}_{\Lambda }^{\left( 0\right) }$ belongs to the interior of $%
\operatorname{Iso}\left( \mathcal{TG}_{\Lambda }\right) $. Now we show the reverse
inclusion. Take $a$ an interior point of $\operatorname{Iso}\left( \mathcal{TG}%
_{\Lambda }\right) $. Then there exits $TZ_{\Lambda }\left( \left. \lambda
\ast _{s}\mu \right\backslash G\right) $ such that $TZ_{\Lambda }\left(
\left. \lambda \ast _{s}\mu \right\backslash G\right) \subseteq \operatorname{Iso}%
\left( \mathcal{TG}_{\Lambda }\right) $ and $a\in TZ_{\Lambda }\left( \left.
\lambda \ast _{s}\mu \right\backslash G\right) $. We have to show $\lambda
=\mu $. Since $a\in TZ_{\Lambda }\left( \left. \lambda \ast _{s}\mu
\right\backslash G\right) $, then $TZ_{\Lambda }\left( \left. \lambda \ast
_{s}\mu \right\backslash G\right) $ is not empty; so $s\left( \lambda
\right) \notin G$. Hence, $\left( \lambda ,d\left( \lambda \right) -d\left(
\mu \right) ,\mu \right) \in TZ_{\Lambda }\left( \left. \lambda \ast _{s}\mu
\right\backslash G\right) $ and since $TZ_{\Lambda }\left( \left. \lambda
\ast _{s}\mu \right\backslash G\right) \subseteq \operatorname{Iso}\left( \mathcal{TG%
}_{\Lambda }\right) $, this implies $\lambda =\mu $. Therefore, $\mathcal{TG}%
_{\Lambda }$ is effective.
\end{proof}

\begin{remark}
Our definition of effectiveness is from Lemma 3.1 of \cite{BCFS14}. That
lemma gives several equivalent characterisations of effective.
\end{remark}

\begin{proposition}
\label{CP-is-isomorphic-to-Steinberg-algebras}Let $\Lambda $ be a row-finite
$k$-graph with no sources, $\mathcal{TG}_{\Lambda }$ be its path groupoid
and $R$ be a commutative ring with $1$. Then there is an isomorphism $\phi
_{Q}:{\normalsize \operatorname{C}}_{R}\left( \Lambda \right) \rightarrow
A_{R}\left( \mathcal{G}_{\Lambda }\right) $ such that $\phi _{Q}\left(
t_{\lambda }\right) =1_{TZ_{\Lambda }\left( \lambda \ast _{s}s\left( \lambda
\right) \right) }$ and $\phi _{Q}\left( t_{\mu ^{\ast }}\right)
=1_{TZ_{\Lambda }\left( s\left( \mu \right) \ast _{s}\mu \right) }$ for $%
\lambda ,u\in \Lambda $.
\end{proposition}

Before proving Proposition \ref{CP-is-isomorphic-to-Steinberg-algebras}, we
first note that the argument of \cite[Lemma 5.6]{CP15} also applies to the
path groupoid $\mathcal{TG}_{\Lambda }$\ and we get the following result.

\begin{lemma}
\label{compact-open-bisection-U-is-in-span}Let $\left\{ TZ_{\Lambda }\left(
\left. \lambda _{i}\ast _{s}\mu _{i}\right\backslash G_{i}\right) \right\}
_{i=1}^{n}$ be a finite collection of compact open bisection sets and%
\begin{equation*}
U:=\bigcup_{i=1}^{n}TZ_{\Lambda }\left( \left. \lambda _{i}\ast _{s}\mu
_{i}\right\backslash G_{i}\right) \text{.}
\end{equation*}%
Then%
\begin{equation*}
1_{U}\in \operatorname{span}_{R}\left\{ 1_{TZ_{\Lambda }\left( \left. \lambda \ast
_{s}\mu \right\backslash G\right) }:\left( \lambda ,\mu \right) \in \Lambda
\ast _{s}\Lambda ,G\subseteq s\left( \lambda \right) \Lambda \right\} \text{.%
}
\end{equation*}
\end{lemma}

\begin{proof}[Proof of Proposition \protect\ref%
{CP-is-isomorphic-to-Steinberg-algebras}]
By \cite[Theorem 6.9]{FMY05} and \cite[Example 7.1]{Y07}, $Q_{\lambda
}:=1_{TZ_{\Lambda }\left( \lambda \ast _{s}s\left( \lambda \right) \right) }$
determines a Toeplitz-Cuntz-Krieger $\Lambda $-family. Now define $%
Q_{\lambda }:=1_{TZ_{\Lambda }\left( \lambda \ast _{s}s\left( \lambda
\right) \right) }$ and $Q_{\mu ^{\ast }}:=1_{TZ_{\Lambda }\left( s\left( \mu
\right) \ast _{s}\mu \right) }$ for $\lambda ,\mu \in \Lambda $. Then $%
\left\{ Q_{\lambda },Q_{\mu ^{\ast }}:\lambda ,u\in \Lambda \right\} $ is a
Cohn $\Lambda $-family in $A\left( \mathcal{TG}_{\Lambda }\right) $. Hence
there exists a homomorphism $\phi _{Q}:{\normalsize \operatorname{C}}_{R}\left(
\Lambda \right) \rightarrow A_{R}\left( \mathcal{TG}_{\Lambda }\right) $
such that $\phi _{Q}\left( t_{\lambda }\right) =Q_{\lambda }$ and $\phi
_{Q}\left( t_{\mu ^{\ast }}\right) =Q_{\mu ^{\ast }}$ for $\lambda ,\mu \in
\Lambda $.

Now we show that $\phi _{Q}$ is injective. Note that for all $r\in \left.
R\right\backslash \left\{ 0\right\} $ and $v\in \Lambda ^{0}$, we have%
\begin{equation*}
\phi _{Q}\left( rt_{v}\right) =rQ_{v}=r1_{TZ_{\Lambda }\left( v\ast
_{s}v\right) }\neq 0
\end{equation*}%
and%
\begin{align*}
\phi _{Q}\Big(r\prod_{e\in v\Lambda ^{1}}\left( t_{v}-t_{e}t_{e^{\ast
}}\right) \Big)& =r\prod_{e\in v\Lambda ^{1}}\left( Q_{v}-Q_{e}Q_{e^{\ast
}}\right) =r\prod_{e\in v\Lambda ^{1}}\left( 1_{TZ_{\Lambda }\left( v\ast
_{s}v\right) }-1_{TZ_{\Lambda }\left( e\ast _{s}e\right) }\right) \\
& =r\prod_{e\in v\Lambda ^{1}}1_{TZ_{\Lambda }\left( \left. v\ast
_{s}v\right\backslash \left\{ e\right\} \right) }=r1_{\prod_{e\in v\Lambda
^{1}}TZ_{\Lambda }\left( \left. v\ast _{s}v\right\backslash \left\{
e\right\} \right) } \\
& =r1_{TZ_{\Lambda }(\left. v\ast _{s}v\right\backslash v\Lambda ^{1})}\neq 0%
\text{.}
\end{align*}%
Hence, by Theorem \ref{the-uniqueness-theorem-of-CP-family}, $\phi _{Q}$ is
injective, as required.

Finally we show the surjectivity of $\phi _{Q}$. Take $f\in A_{R}\left(
\mathcal{TG}_{\Lambda }\right) $. By \cite[Lemma 2.2]{CE-M15}, $f$ can be
written as $\sum_{U\in F}a_{U}1_{U}$ where $a_{U}\in R$, each $U$ is in the
form $\bigcup_{i=1}^{n}TZ_{\Lambda }\left( \left. \lambda _{i}\ast _{s}\mu
_{i}\right\backslash G_{i}\right) $ for some $n\in \mathbb{N}$, and $F$ is
finite set of mutually disjoint elements. Hence to show $f\in \operatorname{im}%
\left( \phi _{Q}\right) $, it suffices to show
\begin{equation*}
1_{U}\in \operatorname{im}\left( \phi _{Q}\right)
\end{equation*}%
where $U:=\bigcup_{i=1}^{n}TZ_{\Lambda }\left( \left. \lambda _{i}\ast
_{s}\mu _{i}\right\backslash G_{i}\right) $ for some $n\in \mathbb{N}$ and
collection $\left\{ TZ_{\Lambda }\left( \left. \lambda _{i}\ast _{s}\mu
_{i}\right\backslash G_{i}\right) \right\} _{i=1}^{n}$. By Lemma \ref%
{compact-open-bisection-U-is-in-span}, $1_{U}$ can be written as the sum of
elements in the form $1_{TZ_{\Lambda }\left( \left. \lambda \ast _{s}\mu
\right\backslash G\right) }$. On the other hand, by following the argument
of \cite[Equation 5.5]{CP15}, we have
\begin{equation*}
1_{TZ_{\Lambda }\left( \left. \lambda \ast _{s}\mu \right\backslash G\right)
}=Q_{\lambda }\big(\prod_{\nu \in G}\left( Q_{s\left( \lambda \right)
}-Q_{\nu }Q_{\nu ^{\ast }}\right) \big)Q_{\mu }
\end{equation*}%
for all $\left( \lambda ,\mu \right) \in \Lambda \ast _{s}\Lambda $ and
finite $G\subseteq s\left( \lambda \right) \Lambda $. Hence, every $%
1_{TZ_{\Lambda }\left( \left. \lambda \ast _{s}\mu \right\backslash G\right)
}$ belongs to $\operatorname{im}\left( \phi _{Q}\right) $ and so does $1_{U}$, as
required. Hence $\phi _{Q}$ is surjective and then is an isomorphism.
\end{proof}

\begin{remark}
Proposition 5.5 of \cite{CP15} shows that the Kumjian-Pask algebra of $%
T\Lambda $ is isomorphic to the Steinberg algebra associated to the
boundary-path groupoid $\mathcal{G}_{T\Lambda }$ of \cite{Y07}. Indeed, we
could have shown that the path groupoid $\mathcal{TG}_{\Lambda }$ of Example %
\ref{groupouid-TGlambda} is topologically isomorphic to the boundary-path
groupoid $\mathcal{G}_{T\Lambda }$, and deduced Proposition \ref%
{CP-is-isomorphic-to-Steinberg-algebras}. However, the direct argument above
takes about the same amount of effort.
\end{remark}


\begin{thebibliography}{99}
\bibitem{A15} G. Abrams,\textit{\ Leavitt path algebras: the first decade},
Bull. Math. Sci. \textbf{5} (2015), 59--120.

\bibitem{Leavitt path algebras} G. Abrams, P. Ara and M. Siles Molina,
Leavitt path algebras, A primer and handbook, Springer, to appear.
https://www.dropbox.com/s/gqqx735jddrip8f/AbramsAraSiles\_BookC1C2.pdf?dl=0

\bibitem{AA05} G. Abrams and G. Aranda Pino, \textit{The Leavitt path
algebra of a graph}, J. Algebra \textbf{293} (2005), 319--334.

\bibitem{AA08} G. Abrams and G. Aranda Pino, \textit{The Leavitt path
algebras of arbitrary graphs,} Houston J. Math. \textbf{34} (2008), 423--442.

\bibitem{AK13} G. Abrams and M. Kanuni, \textit{Cohn path algebras have
Invariant Basis Number}, Commun. Alg. \textbf{44} (2016), 371--380.

\bibitem{AM12} G. Abrams and Z. Mesyan, \textit{Simple Lie algebras arising
from Leavitt path algebras}, J. Pure Applied Algebra \textbf{216} (2012),
2302--2313.

\bibitem{AG11} P. Ara and K.R. Goodearl, $C^{\ast }$\textit{-algebras of
separated graphs}, J. Funct. Anal. \textbf{261} (2011), 2540--2568.

\bibitem{AG12} P. Ara and K.R. Goodearl, \textit{Leavitt path algebras of
separated graphs}, J. Reine Angew. Math. \textbf{669} (2012), 165--224.

\bibitem{AMP07} P. Ara, M.A. Moreno and E. Pardo, \textit{Nonstable }$K$%
\textit{-theory for graph algebras}, Algebr.\ Represent. Theory \textbf{10}
(2007), 157--178.

\bibitem{ACaHR13} G.\ Aranda Pino, J. Clark, A. an Huef and I.\ Raeburn,
\textit{Kumjian-Pask algebras of higher-rank graphs}, Trans. Amer. Math.
Soc. \textbf{365} (2013), 3613--3641.

\bibitem{BCFS14} J. Brown, L.O. Clark, C. Farthing and A. Sims, \textit{%
Simplicity of algebras associated to \'{e}tale groupoids},\ Semigroup Forum
\textbf{88} (2014), 433--452.

\bibitem{CE-M15} L.O. Clark and C. Edie-Michell, \textit{Uniqueness theorems
for Steinberg algebras}, Algebr.\ Represent. Theory \textbf{18} (2015),
907--916.

\bibitem{CFST14} L.O. Clark, C. Farthing, A. Sims and M. Tomforde, \textit{A
groupoid generalisation of Leavitt path algebras}, Semigroup Forum \textbf{89%
} (2014), 501--517.

\bibitem{CFaH14} L.O. Clark, C. Flynn and A. an Huef, \textit{Kumjian-Pask
algebras of locally convex higher-rank graphs}, J. Algebra \textbf{399}
(2014), 445--474.

\bibitem{CP15} L.O. Clark and Y.E.P. Pangalela, \textit{Kumjian-Pask
algebras of finitely aligned higher-rank graphs}, arXiv:1512.06547
[math.RA], 2015.

\bibitem{C55} P.M. Cohn, \textit{Some remarks on the invariant basis property%
}, Topology \textbf{5} (1966), 215--228.

\bibitem{FMY05} C. Farthing, P. S. Muhly and T. Yeend, \textit{Higher-rank
graph }$C^{\ast }$\textit{-algebras: an inverse semigroup and groupoid
approach}, Semigroup Forum \textbf{71} (2005), 159--187.

\bibitem{FLR00} N.J. Fowler, M. Laca and I. Raeburn, \textit{The }$C^{\ast }$%
\textit{-algebras of infinite graphs}, Proc. Amer. Math. Soc. \textbf{128}
(2000), 2319--2327.

\bibitem{FR99} N.J. Fowler and I. Raeburn, \textit{The Toeplitz algebra of a
Hilbert bimodule}, Indiana Univ. Math. J. \textbf{48} (1999), 155--181.

\bibitem{HRSW13} R. Hazlewood, I. Raeburn, A. Sims and S.B.G. Webster,
\textit{Remarks on some fundamental results about higher-rank graphs and
their }$C^{\ast }$\textit{-algebras}, Proc. Edinb. Math. Soc. \textbf{56}
(2013), 575--597.

\bibitem{aHKR15} A. an Huef, S. Kang and I.\ Raeburn, \textit{Spatial
realisations of KMS states on the }$C^{\ast }$\textit{-algebras of
higher-rank graphs}, J. Math. Anal. Appl. \textbf{427} (2015), 977--1003.

\bibitem{aHLRS14} A. an Huef, M. Laca, I. Raeburn and A. Sims, \textit{KMS
states on }$C^{\ast }$\textit{-algebras associated to higher-rank graphs},
J. Funct. Anal. \textbf{266} (2014), 265--283.

\bibitem{aHLRS15} A. an Huef, M. Laca, I. Raeburn and A. Sims, \textit{KMS
states on the }$C^{\ast }$\textit{-algebra of a higher-rank graph and
periodicity in the path space}, J. Funct. Anal. \textbf{268} (2015),
1840--1875.

\bibitem{KP00} A. Kumjian and D. Pask, \textit{Higher rank graph }$C^{\ast }$%
\textit{-algebras}, New York J. Math. \textbf{6} (2000), 1--20.

\bibitem{L62} W.G. Leavitt, \textit{The module type of a ring}, Trans. Amer.
Math.\ Soc. \textbf{42} (1962), 113--130.

\bibitem{LS10} P. Lewin and A. Sims, \textit{Aperiodicity and cofinality for
finitely aligned higher-rank graphs}, Math. Proc. Cambridge Philos. Soc.
\textbf{149} (2010), 333--350.

\bibitem{MT04} P. Muhly and M. Tomforde, \textit{Adding tails to }$C^{\ast }$%
\textit{-correspondences}, Documenta Math. \textbf{9} (2004), 79--106.

\bibitem{P15} Y.E.P. Pangalela, \textit{Realising the Toeplitz algebra of a
higher-rank graph as a Cuntz-Krieger algebra}, New York J. Math. \textbf{22}
(2016), 277--291.

\bibitem{RS05} I. Raeburn and A. Sims, \textit{Product systems of graphs and
the Toeplitz algebras of higher-rank graphs}, J. Operator Theory \textbf{53}
(2005), 399--429.

\bibitem{RSY03} I. Raeburn, A. Sims and T. Yeend, \textit{Higher-rank graphs
and their }$C^{\ast }$\textit{-algebras}, Proc. Edinb. Math. Soc. \textbf{46}
(2003), 99--115.

\bibitem{RSY04} I. Raeburn, A. Sims and T. Yeend, \textit{The }$C^{\ast }$%
\textit{-algebras of finitely aligned higher-rank graphs}, J. Funct. Anal.
\textbf{213} (2004), 206--240.

\bibitem{CBMS} I. Raeburn, Graph Algebras, CBMS Regional Conference Series
in Math., vol. 103, American Mathematical Society, 2005.

\bibitem{RS07} D. Robertson and A. Sims, \textit{Simplicity of }$C^{\ast }$%
\textit{-algebras associated to higher-rank graphs}, Bull. London Math. Soc.
\textbf{39} (2007), 337--344.

\bibitem{RS09} D. Robertson and A. Sims, \textit{Simplicity of }$C^{\ast }$%
\textit{-algebras associated to row-finite locally convex higher-rank graphs}%
, Israel J. Math. \textbf{172} (2009), 171--192.

\bibitem{RS02} G. Robertson and T. Steger,\ \textit{Affine buildings, tiling
systems and higher-rank Cuntz-Krieger algebras}, J. Reine Angew. Math.
\textbf{513} (1999), 115--144.

\bibitem{Sh12} J. Shotwell, \textit{Simplicity of finitely-aligned }$k$%
\textit{-graph }$C^{\ast }$\textit{-algebras}, J. Operator Theory \textbf{67}
(2012), 335--347.

\bibitem{Si03} A. Sims, $C^{\ast }$\textit{-algebras associated to
higher-rank graphs}, PhD Thesis, Univ. Newcastle, 2003.

\bibitem{Si10} A. Sims, \textit{The co-universal }$C^{\ast }$\textit{%
-algebra of a row-finite graph}, New York J. Math. \textbf{16} (2010),
507--524.

\bibitem{St10} B. Steinberg, \textit{A Groupoid approach to discrete inverse
semigroup algebras}, Adv. Math. \textbf{223} (2010), 689--727.

\bibitem{W11} S. B. G. Webster, \textit{The path space of a higher-rank graph%
}, Studia Math., \textbf{204} (2011), 155--185.

\bibitem{Y07} T. Yeend, \textit{Groupoid models for the }$C^{\ast }$\textit{%
-algebras of topological higher-rank graphs}, J. Operator Theory \textbf{51}
(2007), 95--120.
\end{thebibliography}
\end{document}